\documentclass[11pt, twoside]{article} % use larger type; default would be 10pt

\usepackage[utf8]{inputenc} % set input encoding (not needed with XeLaTeX)
\usepackage[margin=1in]{geometry} % to change the page dimensions
\usepackage{graphicx} % support the \includegraphics command and options
\usepackage{booktabs} % for much better looking tables
\usepackage{array} % for better arrays (eg matrices) in maths
\usepackage{paralist} % very flexible & customisable lists (eg. enumerate/itemize, etc.)
\usepackage{verbatim} % adds environment for commenting out blocks of text & for better verbatim
\usepackage{mathrsfs}
\usepackage{amssymb}
\usepackage{amsthm}
\usepackage{amsmath,amsfonts,amssymb}
\usepackage{esint}
\usepackage{graphics}
\usepackage{enumerate}
\usepackage{mathtools}
\usepackage{xfrac}
\usepackage{nicefrac}
\usepackage{subcaption}
\usepackage[normalem]{ulem}
\usepackage{cancel}
\usepackage{enumitem}
\usepackage{dsfont}

\let\olddiamond\diamond
\let\oldsquare\square % Must go before mathabx

\usepackage{mathabx}

\renewcommand{\square}{\oldsquare}
\renewcommand{\diamond}{\olddiamond}

\usepackage[dvipsnames]{xcolor}
\usepackage[colorlinks=true, pdfstartview=FitV, linkcolor=blue, citecolor=blue, urlcolor=blue]{hyperref}
\usepackage[normalem]{ulem}

\usepackage{tikz}
\usetikzlibrary{calc}
\usepackage{pgf}
\usetikzlibrary{external}
\numberwithin{equation}{section}
\numberwithin{figure}{section}

\newtheorem{theorem}{Theorem}[section]

\newtheorem{assumption}[theorem]{Assumption}
\newtheorem{corollary}[theorem]{Corollary}
\newtheorem{proposition}[theorem]{Proposition}
\newtheorem{lemma}[theorem]{Lemma}

\theoremstyle{definition}

\newcommand*{\N}{\ensuremath{\mathbb{N}}}

\newcommand*{\Z}{\ensuremath{\mathbb{Z}}}

\newcommand*{\R}{\ensuremath{\mathbb{R}}}
\newcommand*{\Zd}{\ensuremath{\mathbb{Z}^d}}
\newcommand*{\Rd}{\ensuremath{\mathbb{R}^d}}

\renewcommand*{\tilde}{\widetilde}

\renewcommand{\P}{\ensuremath{\mathbb{P}}}
\renewcommand{\O}{\ensuremath{\mathcal{O}}}
\newcommand{\X}{\ensuremath{\mathcal{X}}}

\newcommand{\qand}{\quad \mbox{and} \quad }

\newcommand{\f}{\mathbf{f}}

\newcommand{\h}{\mathbf{h}}
\newcommand{\s}{\mathbf{s}}

\renewcommand{\S}{\mathcal{S}}

\DeclareMathSymbol{\shortminus}{\mathbin}{AMSa}{"39}
\newcommand{\nf}{\nicefrac}

\DeclareSymbolFont{boldoperators}{OT1}{cmr}{bx}{n}
\SetSymbolFont{boldoperators}{bold}{OT1}{cmr}{bx}{n}
\usepackage{accents}
\newcommand\thickbar[1]{\accentset{\rule{.45em}{.6pt}}{#1}}
\renewcommand{\bar}{\thickbar}
\renewcommand{\a}{\mathbf{a}}
\renewcommand{\k}{\mathbf{k}}

\newcommand{\ahom}{\bar{\a}}

\newcommand{\shom}{\bar{\mathbf{s}}}
\newcommand{\khom}{\bar{\mathbf{k}}}

\renewcommand{\subset}{\subseteq}

\newcommand{\bfA}{\mathbf{A}}
\newcommand{\bfAhom}{\overline{\mathbf{A}}}

\makeatletter 
\newcommand{\negphantom}{\v@true\h@true\negph@nt} 
\newcommand{\neghphantom}{\v@false\h@true\negph@nt} 
\newcommand{\negph@nt}{\ifmmode\expandafter\mathpalette 
	\expandafter\mathnegph@nt\else\expandafter\makenegph@nt\fi} 
\newcommand{\makenegph@nt}[1]{% 
	\setbox\z@\hbox{\color@begingroup#1\color@endgroup}\finnegph@nt} 
\newcommand{\finnegph@nt}{% 
	\setbox\tw@\null 
	\ifv@ \ht\tw@\ht\z@\dp\tw@\dp\z@\fi \ifh@\wd\tw@-\wd\z@\fi\box\tw@} 
\newcommand{\mathnegph@nt}[2]{% 
	\setbox\z@\hbox{$\m@th #1{#2}$}\finnegph@nt} 
\makeatother 

\newcommand{\Hminus}{\hat{\phantom{H}}\negphantom{H}H^{-s}}
\newcommand{\Hminusul}{\hat{\phantom{H}}\negphantom{H}\underline{H}^{-s}}

\def\Xint#1{\mathchoice
	{\XXint\displaystyle\textstyle{#1}}%
	{\XXint\textstyle\scriptstyle{#1}}%
	{\XXint\scriptstyle\scriptscriptstyle{#1}}%
	{\XXint\scriptscriptstyle\scriptscriptstyle{#1}}%
	\!\int}
\def\XXint#1#2#3{{\setbox0=\hbox{$#1{#2#3}{\int}$}
		\vcenter{\hbox{$#2#3$}}\kern-.5\wd0}}

\def\fint{\Xint-}

\newcommand{\avsum}{\mathop{\mathpalette\avsuminner\relax}\displaylimits}

\makeatletter
\newcommand\avsuminner[2]{%
	{\sbox0{$\m@th#1\sum$}%
		\vphantom{\usebox0}%
		\ooalign{%
			\hidewidth
			\smash{\,\rule[.23em]{8.8pt}{1.1pt} \relax}%
			\hidewidth\cr
			$\m@th#1\sum$\cr
		}%
	}%
}
\makeatother

\makeatletter
\newcommand\avsuminnerr[2]{%
	{\sbox0{$\m@th#1\sum$}%
		\vphantom{\usebox0}%
		\ooalign{%
			\hidewidth
			\smash{\,\rule[.23em]{6pt}{0.7pt} \relax}%
			\hidewidth\cr
			$\m@th#1\sum$\cr
		}%
	}%
}
\makeatother

\let\originalleft\left
\let\originalright\right
\renewcommand{\left}{\mathopen{}\mathclose\bgroup\originalleft}
\renewcommand{\right}{\aftergroup\egroup\originalright}

\newcommand{\cu}{\square}
\definecolor{labelkey}{rgb}{0,0,1}

\renewcommand{\hat}{\widehat}

\usepackage{titlesec}

\newcommand{\addperiod}[1]{#1.}
\titleformat{\section}
{\centering\normalfont\Large}{\thesection.}{0.5em}{}
\titleformat*{\subsection}{\bfseries}
\titleformat{\subsubsection}[runin]
{\normalfont\bfseries}
{\thesubsubsection.}
{0.5em}
{\addperiod}
\titleformat*{\subsubsection}{\normalfont\itshape}
\titleformat*{\paragraph}{\bfseries}
\titleformat*{\subparagraph}{\large\bfseries}

\date{\today}

\usepackage[nottoc,notlot,notlof]{tocbibind}
\usepackage{fancyhdr}
\usepackage{extramarks}
\title{High-order Regularity Theory for High-contrast Elliptic Homogenization}
\author{Heikki Lohi\footnote{Department of Mathematics and Statistics, University of Helsinki (heikki.lohi@helsinki.fi)}}
\date{July 29, 2026}

\begin{document}

\maketitle

\begin{abstract}
The purpose of this article is to formulate and prove a global high-order regularity result within the high-contrast framework of elliptic homogenization. In order to achieve this, we also present a version of the high-contrast Caccioppoli inequality.
\end{abstract}

\tableofcontents 
\pagenumbering{arabic}

\section{Introduction}

We are interested in the following second-order linear elliptic partial differential equation that has the divergence-form of
\begin{align}\label{ellipeq}
    -\nabla\cdot\mathbf{a}(x)\nabla u=0
\end{align}
in $U\subset \R^d$ and $d\in \N$. Here, the random coefficient field $\mathbf{a}\colon \R^d\to\R^{d\times d}$ is defined (for not necessarily symmetric matrices) as a Lebesgue measurable $\Z^d$-stationary mapping with merely $L_{\textnormal{loc}}^1$ integrability. We understand the space $L_{\textnormal{loc}}^1(U)$ as the collection of functions $f\in L^1(V)$, where $V\subset U$ is bounded so that $\overline{V}\subset\textnormal{int}\: U$. Naturally, we similarly extend this definition for coefficient fields as well. In a general sense, the overall situation underlying this paper is ultimately the same as in \cite{AK.HC} (or as in \cite{ak} with the addition of the high-contrast setting). However, we will also provide all the necessary preliminaries here, but further details can always be found in \cite{AK.HC}. 

Classically, the theory developed for stochastic homogenization has been from the viewpoint of uniformly elliptic PDEs, whose study was initiated by De Giorgi and Spagnolo in the 1970s (see e.g. \cite{degiorgi} for an English translation). Over the last decade or so, significant progress has been made in this field, particularly on the quantitative side. For example, the exposition \cite{ak} by Armstrong and Kuusi provides an extensive overview of the current state of the field for the uniformly elliptic equation specified below. More precisely, the framework for a significant majority of previously published papers has classically assumed uniform (or moderate in the sense of this paper) ellipticity for~$\mathbf{a}(x)$. This refers to the well-known case, where the random coefficient field $\mathbf{a}(x)$ satisfies (almost surely with respect to the underlying probability measure $\mathbb{P}$) the \emph{uniform ellipticity conditions}
\begin{align}\label{unifellipconds}
    \lambda |e|^2\leq e\cdot \mathbf{a}(x)e\quad \textnormal{and} \quad \frac{|e|^2}{\Lambda}\leq e\cdot \mathbf{a}(x)^{-1}e
\end{align}
for all $e,x\in\R^d$. Above, $0<\lambda\leq\Lambda<\infty$ denote the so-called fixed \emph{ellipticity constants} that bound the eigenvalues of the matrix $\mathbf{a}(x)$. Another important role for these constants is to quantify the ''moderateness'' of the ellipticity contrast, which is determined by the \emph{ellipticity ratio} $\Pi:=\Lambda/\lambda$ given in (\ref{unifellipconds}).

More recently, the interest in \emph{high-contrast elliptic homogenization} (or degenerate elliptic homogenization) has been piqued, and \cite{AK.HC} established a fully rigorous mathematical framework with certain fundamental results for this setting. In high-contrast homogenization, the analysis requires much more complicated methods to handle the possibly violent behavior of the coefficients. In this respect, the so-called \emph{coarse-grained ellipticity} plays a fundamental role. 
To define the coarse-grained matrices for the coefficient field $\a(\cdot)$, we first set for every $x\in\R^d$ that 
\begin{align}\label{defbiga}
    \mathbf{A}(x):=\begin{bmatrix}
(\mathbf{s}+\mathbf{k}^t\mathbf{s}^{-1}\mathbf{k})(x) & -(\mathbf{k}^t\mathbf{s}^{-1})(x) \\
-(\mathbf{s}^{-1}\mathbf{k})(x) & \mathbf{s}^{-1}(x)
    \end{bmatrix}\in \R^{2d\times 2d},
\end{align}
where we denote the symmetric and antisymmetric parts by
\begin{align}\label{matrixdecomp}
    \mathbf{s}(x):=\frac{1}{2}(\mathbf{a}(x)+\mathbf{a}^t(x))\in\R_{\textnormal{sym}}^{d\times d}\quad \textnormal{and}\quad \mathbf{k}(x) :=\frac{1}{2}(\mathbf{a}(x)-\mathbf{a}^t(x))\in\R_{\textnormal{anti}}^{d\times d}
\end{align}
in which $\mathbf{a}^t$ is the transpose matrix of $\mathbf{a}$ as well as
\begin{align*}
    \R_{\textnormal{sym}}^{d\times d}:=\{A\in\R^{d\times d}\; |\; A=A^t\}\quad \textnormal{and}\quad \R_{\textnormal{anti}}^{d\times d}:=\{A\in\R^{d\times d}\; |\; A=-A^t\}.
\end{align*}
We collect the coarse-grained matrices into a single~$2d\times2d$ symmetric non-negative matrix~$\bfA(U)$. It has the following variational interpretation (see~\cite[Chapter 2]{AK.HC} for further details). Namely, we have, for every~$P \in \R^{2d}$, the formula that 
\begin{equation}
\label{e.J.P0.Dirichlet}
\frac12 P \cdot \bfA(U) P 
=
\inf\biggl\{ 
\fint_{U} 
\frac12 (X + P) \cdot \bfA(x) (X + P)
\; | \; 
X \in  L^{2}_{\a,\mathrm{pot},0}(U) \times  L^{2}_{\a,\mathrm{sol},0}(U)
\biggr\},
\end{equation}
where~$L^{2}_{\a,\mathrm{pot},0}(U)$ is defined as the closure of the set~$\{  \nabla \phi \; | \; \phi \in C_{\mathrm{c}}^\infty(U)  \}$ in terms of the norm~$\f \mapsto ( \int_{U} \f\cdot \s \f )^{\nicefrac12}$, and~$L^{2}_{\a,\mathrm{sol},0}(U)$ is the closure of~$
\{  \f  \; | \; \f \in C_{\mathrm{c}}^\infty(U;\Rd)\,, \; \nabla \cdot \f = 0  \bigr\}$ in terms of the norm~$\f \mapsto ( \int_{U} \f\cdot \s^{-1} \f )^{\nicefrac12}$. For further properties, we refer to the extensive discussion in~\cite[Chapter 2]{AK.HC}. Especially, we note by testing the variational formulation of~(\ref{e.J.P0.Dirichlet}) with constant vector fields that the quantity~$\bfA(U)$ is bounded under the assumption
\begin{equation*} 
%\label{e.}
\mathbf{s},\mathbf{s}^{-1},\mathbf{k}^t\mathbf{s}^{-1}\mathbf{k}\in L^1(U;\R^{d\times d}).
\end{equation*}

The axiomatic assumption we make for the entire article is that the coarse-grained matrices on suitable scales are controlled by their homogenized limits. 

\begin{assumption} 
\label{d.renormellipt}
We impose the following assumptions throughout the paper.
\begin{itemize}
\item Let $\theta \in (0,1]$ and $\gamma\in(0,1)$ be universally fixed parameters.
\item Let $\bfAhom \in \R_{\mathrm{sym}}^{2d\times 2d}$ be a homogenized positive definite matrix found in (\ref{defhoms}), whose entries are given by a positive definite symmetric matrix $\shom \in \R_{\mathrm{sym}}^{d\times d}$ and an antisymmetric matrix $\khom \in \R_{\mathrm{anti}}^{d\times d}$ as
\begin{align}\label{defbigahom}
    \bfAhom:=\begin{bmatrix}
\shom+\khom^t\shom^{-1}\khom & - \khom^t\shom^{-1} \\
-\shom^{-1}\khom   & \shom^{-1}
    \end{bmatrix}.
\end{align}  
\item Let $\mathbf{q}_0$ be a positive definite symmetric matrix defined in (\ref{defqzero.proof}) that constitutes an adapted universal geometry as follows. For every $n\in \Z$, we denote that $\mathbb{L}_0:=\mathbf{q}_0(\Z^d)$ and define the adapted cubes by
\begin{equation} 
\label{e.adapted}
\diamondsuit_n 
:= 
\mathbf{q}_0 \bigl((-\tfrac12 3^{n} , \tfrac12 3^n)^d\bigr)=\{ x\in\R^d\; |\; \mathbf{q}_0^{-1}x\in (-\tfrac12 3^{n} , \tfrac12 3^n)^d \}.
\end{equation}
\item
Let $\X_{\mathcal{H}}$ be a random variable associated with a $\Z^d$-stationary probability measure~$\P$ that is assumed to be weakly elliptic in the sense of (\ref{P2cond}) and satisfies the mixing condition of (\ref{P3cond}). We suppose that~$\P$ acts as the law of the Lebesgue measurable coefficient field $\mathbf{a}\colon \R^d\to\R^{d\times d}$ with respect to the $\sigma$-algebra of (\ref{defsigalg}) in the sample space
\begin{align*}
    \Omega:=\{\mathbf{a}\in L_{\textnormal{loc}}^1\; |\; \mathbf{a}(x)=\mathbf{s}(x)+\mathbf{k}(x);\;  \mathbf{s},\mathbf{s}^{-1},\mathbf{k}^t\mathbf{s}^{-1}\mathbf{k}\in L_{\textnormal{loc}}^1\}.
\end{align*}
Additionally, let~$\Psi_{\mathcal{H}}\colon\R_+ \to [1,\infty)$ be an increasing function satisfying for all $t \in (0,\infty)$ that
\begin{equation*}
\P\bigl[ \X_{\mathcal{H}} > t \bigr] \leq \frac1{\Psi_{\mathcal{H}}(t)} \quad 
\qand \quad
\lim_{t \to \infty}\Psi_{\mathcal{H}}(t) = \infty.
\end{equation*}
\item We assume for every~$n,m \in \Z$ satisfying~$3^m \geq \X_{\mathcal{H}}$ with~$n \leq m$ and~$z \in 3^n\mathbb{L}_0 \cap \diamondsuit_m$ that 
\begin{align}\label{renormellipt}
\mathbf{A}(z+\diamondsuit_n)\leq \biggl(1+3^{\gamma(m-n)}\Bigl(\frac{\mathcal{X}_{\mathcal{H}}}{3^m} \Bigr)^{\! \theta} \biggr)\mathbf{\overline{A}}.
\end{align}
\end{itemize}
\end{assumption}
Above, we can universally assume without any loss of generality that $\khom=0$. If that is not the case, we can always subtract $\khom$ from both $\a$ and $\ahom$ without altering the set of solutions. This property is very useful for us in our proofs, where we can set for the homogenized matrix $\ahom$ that $\ahom=\shom$.
We also utilized the Loewner ordering notation $A\leq B$ for two symmetric matrices of the same dimension, which means that the difference matrix $B-A$ has non-negative eigenvalues. Each of the objects shown above is discussed more thoroughly later on. Especially, the probability measure $\P$ satisfies the aforementioned conditions to be labeled as (P1)--(P3), which we will formulate more thoroughly in the next chapter. We will also demonstrate in Proposition \ref{p.Psimplyrenormellipt} that these assumptions (P1)--(P3) imply the condition of (\ref{renormellipt}). Given any~$\shom$, we set for its spectral norm that
\begin{equation}\label{deflambdabars}
\overline{\Lambda} := |\shom|,
\quad
\overline{\lambda} := |\shom^{-1}|^{-1}, \qand \Pi_{\overline{\mathbf{s}}}:=\nf{\overline{\Lambda}}{\overline{\lambda}}.
\end{equation}

In Assumption~\ref{d.renormellipt}, we called~$\bfAhom$ the~\emph{homogenized matrix}. Let us briefly discuss why this notion is justified. We start by decomposing~$\bfA(U)$ in the following block form of
\begin{equation}
\label{e.bigA.def}
\bfA(U)
:= 
\begin{bmatrix} 
( \s + \k^t\s_*^{-1}\k )(U) 
& -(\k^t\s_*^{-1})(U) 
\\ - ( \s_*^{-1}\k )(U) 
& \s_*^{-1}(U) 
\end{bmatrix}
\end{equation}
for the coarse-graining matrices~$\s(U)$,~$\s_*^{-1}(U)$, and $\k(U)$ defined by the identity in \eqref{defcgmatrices}. Consequently, 
the assumption~\eqref{renormellipt} and the definitions above imply for $\khom=0$ and for every~$n,m \in \Z$ with~$3^m \geq \mathcal{X_H}$ and~$n \leq m$ as well as~$z \in 3^n\mathbb{L}_0 \cap \diamondsuit_m$ that
\begin{equation*}
( \s+ \k^t \s_*^{-1} \k )(z+\diamondsuit_n)  
\leq
\biggl(1+3^{\gamma(m-n)}\Bigl(\frac{\mathcal{X}_{\mathcal{H}}}{3^m} \Bigr)^{\! \theta} \biggr) \shom,
\end{equation*}
and
\begin{equation*} 
\s_*^{-1}(z+\diamondsuit_n)  
\leq
\biggl(1+3^{\gamma(m-n)}\Bigl(\frac{\mathcal{X}_{\mathcal{H}}}{3^m} \Bigr)^{\! \theta} \biggr) \shom^{-1}.
\end{equation*}
We also note that $\s_*(z+\diamondsuit_n)\leq \s(z+\diamondsuit_n)$ always holds by~\cite[Lemma 5.2]{ak} that utilizes a certain duality argument. Therefore, we have almost surely with respect to $\mathbb{P}$ that
\begin{equation}\label{defhoms}
\lim_{n\to \infty} \s_*(\diamondsuit_n) 
=
\lim_{n\to \infty} \s(\diamondsuit_n)  
=
\shom 
\qand
\lim_{n\to \infty} \k(\diamondsuit_n)  = \khom
\end{equation}
while it follows that~$\lim_{n\to \infty} \bfA(\diamondsuit_n)  = \bfAhom$ almost surely. Note also that one can formulate the same objects with the standard Euclidean cubes as well, because it is always possible to find a bound for the coarse-grained matrices of adapted cubes in terms of the assumption (P2). For a proof of this fact, see~\cite[Lemma 2.13]{AK.HC}.

In this paper, we focus particularly on the quantitative aspects of homogenization theory (for a modern exposition in the periodic case, see e.g. \cite{shen}). Our aim is to generalize the main high-order regularity result of \cite[Theorem 6.12]{ak} and \cite[Theorem 3.8]{akm} to the \emph{high-contrast} framework. The arguments there are, in turn, based on the works of \cite{AKM1} and \cite{AKM2}. The study of regularity theory has a long and versatile history within the field of PDEs. For the purposes of stochastic homogenization theory, these considerations in the periodic setting originated in the late 1980s by Avellaneda and Lin in \cite{avellaneda} as well as in \cite{lin}. In addition, \cite{asmart} made significant contributions to this field, especially under stronger ergodicity assumptions, by introducing the coarse-grained matrices. Many quantitative results of regularity theory were also obtained in a great qualitative generality by \cite{gloria}.

The study of high-contrast homogenization is currently a very active line of research initiated by the seminal article of \cite{AK.HC}. Developing an extensive theory of high-contrast stochastic homogenization is important because it could have further
applications to certain problems in probability, analysis, and mathematical physics. A great example of this potential was seen in the recent proof of a superdiffusive central limit theorem in~\cite{ABK.SD} that relies heavily on the high-contrast theory introduced in \cite{AK.HC}.

The role of the \emph{Caccioppoli inequality} is paramount in the elliptic regularity theory. In~\cite[Proposition 2.5]{AK.HC}, a general high-contrast version of this result is presented. However, while it portrays the involved ellipticity quantities accurately in the general case, we will provide a simplified non-trivial version of it adjusted to our situation. Its proof is also independent of the argument they utilize. We can see its usefulness concretely in Chapter 4, where we will continuously utilize it in our presented proof for Theorem~\ref{thmhighreg}.

Let us then briefly present the main results of this article. There are essentially two primary results, which we aim to highlight here. The first result is the aforementioned version of the high-contrast Caccioppoli inequality.
\begin{proposition}\label{introCI}
Suppose that Assumption \ref{d.renormellipt} holds and that $\|\mathbf{s}^{1/2}\nabla u\|_{\underline{L}^2(\diamondsuit_{m-1})}<\infty$ with $u$ being a solution of (\ref{ellipeq}).
 Then, there exists a constant $C(d,\gamma)<\infty$ for every~$m\in \N$ with~$3^m \geq \X_{\mathcal{H}}$ such that
\begin{equation}\label{introresult1}
\|\mathbf{s}^{1/2}\nabla u \|_{\underline{L}^2(\diamondsuit_{m-1})}
\leq 
C\overline{\lambda}^{1/2}3^{-m}\|u\|_{\underline{L}^2(\diamondsuit_{m})}.
\end{equation}
\end{proposition}
\noindent Here, $\underline{L}^2$ denotes the rescaled $L^2$ space, which is volume-normalized for every $p\in[1,\infty)$ as we define that $\|u\|_{\underline{L}^p(U)}:=|U|^{-1/p}\|u\|_{L^p(U)}$, where~$|U|$ denotes the Lebesgue measure of $U\subset\R^d$.

The second main result of this article is the high-contrast version of the high-order regularity theorem of \cite[Theorem 6.12]{ak}, which is a Liouville-type theorem about entire solutions having at most polynomial growth at infinity. Having a basic high-order regularity theorem such as this one is essential when trying to achieve optimal quantitative estimates in high contrast. We start by defining the space 
\begin{equation*}
    \mathcal{A}_k:=\left\{u\in H_{\s,\mathrm{loc}}^1(\Rd)\; | \;  -\nabla \cdot \a \nabla u = 0 \; \mbox{in } \Rd \qand \; \limsup_{n \to\infty} 3^{-n(k+1)}\|u\|_{\underline{L}^2(\diamondsuit_{n})}=0 \right\}
\end{equation*}
for every $k\in\N$ and, analogously, for the solutions of the homogenized equation, we define that
\begin{equation*}
\overline{\mathcal{A}}_k
:=
\left\{\bar u\in H_{\mathrm{loc}}^{1}(\Rd)\; | \;  -\nabla \cdot \shom \nabla \bar u = 0 \; \mbox{in } \Rd \qand \limsup_{n \to\infty} 3^{-n(k+1)}\|\bar u\|_{\underline{L}^2(\diamondsuit_{n})} =0 \right\}.
\end{equation*}
Here, the functions in $\overline{\mathcal{A}}_k$ are, in fact, $\shom$-harmonic polynomials of degree $k$ or less. The local weighted Sobolev spaces utilized above are defined as follows. Namely, for all bounded Lipschitz domains~$U\subset\R^d$, we set that
\begin{equation*}
H_{\s}^1(U) := \bigl\{ u \in W^{1,1}(U) \; | \; \| u \|_{L^1(U)} + \| \s^{\nf12} \nabla u \|_{L^2(U)} < \infty   \bigr\},
\end{equation*}
and then, rather naturally,
\begin{align*}
H_{\s,\mathrm{loc}}^1(\Rd) = \{ u \in L_{\mathrm{loc}}^1(\Rd) \; | \; \mbox{$u \in H_{\s}^1(U)$ whenever~$U$ is a bounded Lipschitz domain}\}. 
\end{align*}
Lastly, let us recall the volume-normalized seminorm of the Sobolev space~$\Hminus(\diamondsuit_m)$, which is the dual space of $H^s(\diamondsuit_m)$. We will begin with the definitions of the $\underline{H}^s$ norm and seminorm for $s\in(0,1)$, which are classically defined by
\begin{align*}
    \|u\|_{\underline{H}^s(U)}:=\left(|U|^{-\nf{2s}d}\|u\|_{\underline{L}^2(U)}^2 + [u]_{\underline{H}^s(U)}^2\right)^{\nf12} \qand [u]_{\underline{H}^s(U)}:=\left(\fint_U \int_U \frac{|u(x)-u(y)|^2}{|x-y|^{d+2s}}\; dx\: dy \right)^{\nf12}.
\end{align*}
Then, we define the volume-normalized norm and seminorm of $\Hminus$ as follows. Note that these norms coincide due to $C^\infty(U)$ being dense in $H^s(U)$, but we will make this distinction purely for notational reasons so that
\begin{align*}
    \|u\|_{\Hminusul(U)}:=\sup\left(\fint_U uv\; |\; v\in H^s(U), \|v\|_{\underline{H}^s(U)}\leq 1 \right)
\end{align*}
and
\begin{align*}
    [u]_{\Hminusul(U)}:=\sup\left(\fint_U uv\; |\; v\in C^\infty(U), \|v\|_{\underline{H}^s(U)}\leq 1 \right).
\end{align*}

\begin{theorem}\label{thm13}
Let $k\in\N$ and~$s\in(\nf\gamma2,\nf12)$. Suppose that Assumption \ref{d.renormellipt} holds. There exist a constant $C(d,k,s,\gamma)<\infty$ and a positive matrix $\overline{\mathbf{A}}\in\R_\textnormal{sym}^{2d\times 2d}$ as in (\ref{defbigahom}) such that for every $u\in\mathcal{A}_k$ and $m \in \N$ with $3^m \geq \X_{\mathcal{H}}$, there exists some~$\overline{u}\in\overline{\mathcal{A}}_k$, for which we have the estimate that
\begin{equation}\label{introresult2}
\begin{aligned}
   3^{-m}\overline{\lambda}^{1/2}\|u-\overline{u} \|_{\underline{L}^2(\diamondsuit_m)} &+ 3^{-ms}\left[\overline{\mathbf{A}}^{1/2}\begin{bmatrix}
\nabla u-\nabla \overline{u} \\ \mathbf{a}\nabla u-\ahom\nabla\overline{u}
    \end{bmatrix} \right]_{\Hminusul(\diamondsuit_m)}\\
    &\leq C
\left(\frac{\X_{\mathcal{H}}}{3^m} \right)^{\! \nf{\theta}{2}} \min\left\{\|\s^{1/2} \nabla u \|_{\underline{L}^2(\diamondsuit_m)},\; \|\shom^{1/2} \nabla \overline{u} \|_{\underline{L}^2(\diamondsuit_m)}\right\}.
\end{aligned}
\end{equation}
Conversely, for every~$\overline{u}\in\overline{\mathcal{A}}_k$, there exists~$u\in\mathcal{A}_k$ such that (\ref{introresult2}) is valid for every~$m \in \N$ with $3^m \geq \X_{\mathcal{H}}$. 
\end{theorem}

Let us then explain the conventions utilized within this article regarding the usage of constants. Essentially, these conventions are similar to \cite{AK.HC}, that is, none of the constants within the high-contrast context can directly depend on the ellipticity constants or their ratios. Typically, the capital letter $C$ is reserved for constants greater than one, whereas the lowercase $c$ refers to (positive) constants smaller than one. We will also be utilizing the following abbreviation convention to indicate the dependencies of a certain constant. Namely, if we have a constant $C<\infty$ dependent on, for example, the dimension $d\in\N$ and some exponent $\xi\in(0,1]$, then we will indicate this simply by writing $C(d,\xi)<\infty$. On the other hand, if these dependencies have not been explicitly written down, then it means that they should be clear from the context otherwise (e.g. they remain the same as before or transfer directly from the claim). The constants may also vary from line to line.

To conclude this first introductory chapter, we will provide the outline of this paper. In the next chapter, we will present important preliminaries, such as definitions, notation, and basic results, in high contrast. In Chapter 3, we will start working towards the high-order regularity theory in high contrast that was not studied in \cite{AK.HC}. This will be done by formalizing a version of the Caccioppoli inequality in our somewhat more specific setting that simplifies the version presented in \cite[Proposition 2.5]{AK.HC}. Developing this powerful tool allows us to prove a high-order regularity theorem while following the arguments of \cite{ak} and \cite{akm} applied to the high-contrast framework. This will comprise the last chapter of this paper. Finally, throughout the fourth chapter, we will be working with harmonic polynomials within the geometry specifically constructed by the matrix~$\mathbf{q}_0$ that is adapted from the Euclidean geometry, as we explain more thoroughly in the following chapter. The definitions and properties of these polynomials in this geometry are studied and listed in Appendix~A for the reader's convenience.

\section{Preliminaries}

\subsection*{Definitions, notations, and assumptions}

The purpose of this section is to establish a mathematically rigorous foundation for high-contrast ellipticity to be utilized throughout the entire paper. This includes numerous fundamental definitions and fixing some of the notation for later. We aim to maintain these definitions and notations as consistent as possible with respect to \cite{AK.HC}. Furthermore, during this section, we will also provide all the underlying assumptions for the probability measure $\mathbb{P}$, whose realizations the coefficient matrices $\mathbf{a}(x)$ are in the probability space introduced below. Some of the necessary preliminaries were already presented in the previous chapter, and thus we will mostly not repeat them here. However, we might specify certain aspects further within this chapter.

We will start by examining the set
\begin{align*}
    \R_+^{d\times d}:=\{A\in\R^{d\times d}\; |\; e\cdot Ae\geq 0 \; \textnormal{for every $e\in\R^d$} \}
\end{align*}
of (not necessarily symmetric) square matrices and define the set of coefficient fields precisely by
\begin{align}\label{omegadef}
\Omega:=\{\mathbf{a}\in L_{\textnormal{loc}}^1(\R^d;  \R_+^{d\times d})\; |\; \mathbf{a}=\mathbf{s}+\mathbf{k}\; \textnormal{so that}\; \mathbf{s},\mathbf{s}^{-1},\mathbf{k}^t\mathbf{s}^{-1}\mathbf{k}\in L_{\textnormal{loc}}^1(\R^d;\R^{d\times d})\}.
\end{align}
Above, we require, of course, that $\mathbf{a}(x)=\mathbf{s}(x)+\mathbf{k}(x)$ holds for almost every $x\in\R^d$. Note that the aforementioned conditions always imply the existence of a unique inverse matrix $\mathbf{a}^{-1}(x)\in\R_+^{d\times d}$, because the null space of $\mathbf{a}(x)$ is trivial. Of course, this means that $\mathbf{s}^{-1}$ is also well-defined in the definition of (\ref{omegadef}). We also recall from basic matrix algebra that, indeed, every square matrix~$\mathbf{a}(x)$ has the sum decomposition $\mathbf{a}(x)=\mathbf{s}(x)+\mathbf{k}(x)$ given by (\ref{matrixdecomp}).

For all Borel sets $U\subset \R^d$, we define the $\sigma$-algebra $\mathcal{F}(U)$ to be generated via the random variables
\begin{align}\label{defsigalg}
    \mathbf{a}\mapsto \int_{\R^d} e'\cdot \mathbf{a}(x)e\varphi(x)\; dx
\end{align}
for each $e,e'\in\R^d$ and each test function $\varphi\in C_c^\infty(U)$. We will utilize the abbreviation that $\mathcal{F}:=\mathcal{F}(\R^d)$ from now on. Consequently, we wish to equip $\mathcal{F}$ with a group action that involves all $\R^d$-translations. Namely, the group $G:=\{T_y\; |\; y\in\R^d\}$ of $\R^d$-translations in $\Omega$ (that is, the translation mappings $T_y\colon \Omega\to\Omega$, $T_y\mathbf{a}=\mathbf{a}(\cdot+y)$ as $y\in\R^d$ is fixed) can be extended to the $\sigma$-algebra~$\mathcal{F}$ in a natural way by defining that $T_yF:=\{T_y\mathbf{a}\; |\; \mathbf{a}\in F\}$ as $F\in \mathcal{F}$.

Note that we could have equivalently defined $\Omega$ based on the definition of $\mathbf{A}(\cdot)$ instead of~$\mathbf{a}(\cdot)$ that was utilized in (\ref{omegadef}). Consequently, we may treat $\mathbf{a}(\cdot)$ and $\mathbf{A}(\cdot)$ interchangeably depending on the task at hand.

It is necessary for us to identify the function spaces and their norms in which we will perform our upcoming analysis. Again, for further details and remarks, the interested reader should see~\cite{AK.HC}. The weighted Sobolev space $H_\mathbf{s}^1(U)$ for $U\subset\R^d$ plays a prominent role in all of our analysis as the underlying function space, which we have already seen in the introductory chapter. Namely, here we define a bit more precisely that $H_\mathbf{s}^1(U)$ is the completion of the space $C^\infty(U)$, where the norm is given (up to an additive constant) by
\begin{align}\label{weightednorm}
\|u \|_{H_\mathbf{s}^1(U)}:=\left(\|u \|_{L^1(U)}^2 + \int_U \nabla u\cdot \mathbf{s}\nabla u \right)^{\nf{1}{2}}.
\end{align}

Naturally, $H_\mathbf{s}^1(U)$ has a compactly supported counterpart $H_{\mathbf{s},c}^1(U)$, which is defined as the closure of the test function space $C_c^\infty(U)$ equipped with the norm in (\ref{weightednorm}). The space~$H_{\mathbf{s}}^1(U)$ has the following important subspace
\begin{align*}
    \mathcal{A}(U):=\{u\in H_\mathbf{s}^1(U)\; |\; -\nabla\cdot \mathbf{a}\nabla u=0\; \textnormal{in $U$}\},
\end{align*}
whose condition is interpreted distributionally so that
\begin{align*}
    -\fint_U \nabla\varphi\cdot\a\nabla u=0
\end{align*}
for every $\varphi\in C_c^\infty(U)$. We will also denote by $\overline{\mathcal{A}}(U)$ the same space as above, but instead for the solutions $u\in H^{1}(U)$ of the homogenized equation $-\nabla\cdot\ahom\nabla u=0$. Let us also define the dual space of $H_{\mathbf{s},c}^1(U)$, which we will signify by
\begin{align*}
    H_\mathbf{s}^{-1}(U):=\{\nabla\cdot\mathbf{s}^{\nf 12}\mathbf{f}\; |\; \mathbf{f}\in L^2(U)^d \}.
\end{align*}
We will equip $H_\mathbf{s}^{-1}(U)$ with the classical dual norm as
\begin{align*}
    \|f \|_{H_\mathbf{s}^{-1}(U)}:=\sup\left\{(u,f)\; |\; u\in H_{\mathbf{s},c}^1(U),\, \|u\|_{H_\mathbf{s}^{1}(U)}\leq 1 \right\}.
\end{align*}
Lastly, we define the space $H_\mathbf{a}^1(U)$ as the closure of $C^\infty(U)$ with respect to the norm
\begin{align*}
    \|u\|_{H_\mathbf{a}^1(U)}:=\left(\|u\|_{H_\mathbf{s}^1(U)}^2+\|\nabla\cdot\mathbf{k}\nabla u\|_{H_\mathbf{s}^{-1}(U)}^2 \right)^{\nf{1}{2}}.
\end{align*}

Also, we recall from \cite{ak} or \cite{akm} that each bounded Lipschitz domain $U\subset\R^d$ has the \emph{coarse-grained matrix} of $\mathbf{A}(U)$ and its invertible dual of $\mathbf{A}_*(U)$ even in the high-contrast setting. These matrices were introduced in \cite{asmart} for the first time and subsequently further developed in \cite{ak}.
The exact definitions of $\mathbf{s}(U)$, $\mathbf{s}_*(U)$, and $\mathbf{k}(U)$ utilize the variational quantities of $J(U,p,q)$ and $J^*(U,p,q)$ defined for every
$p,q\in\R^d$ by the integral averages
\begin{align}\label{defJUpq}
    J(U,p,q):=\sup_{v\in\mathcal{A}(U)} \fint_{U} \left(-\frac{1}{2}\nabla v\cdot \mathbf{s}\nabla v - p\cdot \mathbf{a}\nabla v + q\cdot \nabla v \right)
\end{align}
as well as
\begin{align}\label{defJStarUpq}
    J^*(U,p,q):=\sup_{v\in\mathcal{A}^*(U)} \fint_{U} \left(-\frac{1}{2}\nabla v\cdot \mathbf{s}\nabla v - p\cdot \mathbf{a}^t\nabla v + q\cdot \nabla v \right)
\end{align}
with $\mathcal{A}^*(U):=\{ v\in H_\mathbf{a}^1(U)\; |\; -\nabla\cdot\mathbf{a}^t\nabla v=0 \}$.
We then define $\mathbf{s}(U), \mathbf{s}_*(U)\in\R_\textnormal{sym}^{d\times d}$ and $\mathbf{k}(U)\in\R^{d\times d}$ variationally so that
\begin{align}\label{defcgmatrices}
    J(U,p,q)=\frac{1}{2}p\cdot \mathbf{s}(U)p+\frac{1}{2}(q+\mathbf{k}(U)p)\cdot\mathbf{s}_*^{-1}(U)(q+\mathbf{k}(U)p)-p\cdot q
\end{align}
or, equivalently,
\begin{align}
    J^*(U,p,q)=\frac{1}{2}p\cdot \mathbf{s}(U)p+\frac{1}{2}(q-\mathbf{k}(U)p)\cdot\mathbf{s}_*^{-1}(U)(q-\mathbf{k}(U)p)-p\cdot q.
\end{align}
These objects have a very general and rich structure with many useful characteristics. For example, they satisfy multiple convenient variational properties and ordering properties. Most importantly, they satisfy the \emph{subadditivity} property of \eqref{subaddproperty}, making them the natural choice for studying high-contrast homogenization. We refer to \cite[Chapter 5]{ak} and \cite[Chapter 2]{AK.HC} for more details and motivation.

Our next goal is to equip a certain probability measure $\mathbb{P}$ in the space $(\Omega,\mathcal{F})$, which in turn produces the underlying probability space $(\Omega,\mathcal{F},\mathbb{P})$ for us. The conditions that we require from~$\mathbb{P}$ are paramount for the high-contrast setting, and thus they dictate our inference quite a lot. On the one hand, we must have some ellipticity and ergodicity, but not too much of them either, which would restrict our analysis a bit too heavily. The following three assumptions for $\mathbb{P}$ that we present below are identical (modulo small notational changes) to the ones imposed in \cite{AK.HC}, where we denote their equivalent condition of (P2\textdagger) by simply (P2). Let us first list these assumptions and then discuss them further (while again, for more detailed explanations, motivation, and remarks, we refer to~\cite{AK.HC}). However, we will start by introducing a few more pieces of notation before that.

We denote the \emph{triadic subcubes} by $\square_m$ for every $m\in\Z$ and define them as
\begin{align*}
    \square_m:=\left(-\frac{1}{2}3^m,\frac{1}{2}3^m \right)^d\subset\R^d.
\end{align*}
Another notation that we have not yet explained is the \emph{Malliavin derivative} (at least in some sense), which we denote by $|D_{z+\square_n}X_z|$ in (P3) below. Namely, as long as a random variable $X$ is $\mathcal{F}$-measurable in $\Omega$, we define for each $\mathbf{A}\in\Omega$ that
\begin{equation*}
\begin{aligned}
    |D_U X|(\mathbf{A})
    :=\limsup_{t\to 0} \frac{\sup\{X(\mathbf{A}_1)-X(\mathbf{A}_2)\; |\; \mathbf{A}_1,\mathbf{A}_2\in\Omega, |\mathbf{A}^{-\nf{1}{2}}\mathbf{A}_i\mathbf{A}^{-\nf{1}{2}}-I_{2d}|\leq t1_U, \forall i\in\{1,2\} \}}{2t}.
\end{aligned}
\end{equation*}
Here, $I_{2d}\in\R^{2d\times 2d}$ signifies the identity matrix and $1_U$ signifies the indicator function in $U$. The heuristic interpretation for the Malliavin derivative here is that it measures how strongly $X$ depends on the values of $\a|_U$. Lastly, for a finite set $S$ that consists of some $k\in\N$ elements $s_1,\ldots, s_k$, we will denote its cardinality by $|S|=k$ and, furthermore, we utilize the following notation for their mean that
\begin{align*}
    \avsum_{i=1}^k s_i:=\frac{1}{k}\sum_{i=1}^k s_i.
\end{align*}
\begin{itemize}
\item[(P1)] We assume $\mathbb{P}$ to be $\Z^d$-stationary (i.e. statistically homogeneous), which means that $\mathbb{P}\circ T_z=\mathbb{P}$ for all lattice points $z\in\Z^d$.
\item[(P2)] We assume that the equation in (\ref{ellipeq}) is (weakly) elliptic by having deterministic bounds for the coarse-grained coefficient field on sufficiently large scales. More precisely, we suppose the existence of a random variable $\mathcal{S}\colon \Omega\to [0,\infty)$, which acts as the minimal scale for ellipticity. Namely, the probability of the event $\{\mathcal{S}>s\}$ has the following upper bound for every $s\in(0,\infty)$ that
\begin{align*}
    \mathbb{P}[\mathcal{S}>s]\leq \frac{1}{\Psi_\mathcal{S}(s)}
\end{align*} 
for some increasing function $\Psi_\mathcal{S}\colon \R_+\to [1,\infty)$ with a constant $K_{\Psi_\mathcal{S}}\in(1,\infty)$ such that they together fulfill the requirement that
\begin{align*}
    t\Psi_\mathcal{S}(t)\leq \Psi_\mathcal{S}(tK_{\Psi_\mathcal{S}})
\end{align*}
for every $t\in[1,\infty)$.
Lastly, we assume that there exist a matrix $\mathbf{E}_0\in\R_{\textnormal{sym}}^{2d\times 2d}$ and an exponent $\gamma'\in[0,1)$ as follows. We suppose that the scale $\mathcal{S}$ is chosen so that for every $m\in\Z$, if it holds that $3^m\geq \mathcal{S}$, then for all $n\in (-\infty,m]\cap\Z$ and $z\in 3^n\Z^d\cap \square_m$, we have that
\begin{align}\label{P2cond}
    \mathbf{A}(z+\square_n)\leq 3^{\gamma'(m-n)}\mathbf{E}_0.
\end{align}

\item[(P3)] We assume that a suitable mixing condition, called \emph{concentration for sums} (henceforth abbreviated as CFS), holds, which quantifies the ergodicity involved. More precisely, we suppose the existence of a parameter $\beta\in[0,1)$, an exponent $\nu\in(0,\nf{d}{2}]$, and an increasing function $\Psi\colon \R_+\to [1,\infty)$ with constant $K_\Psi\in [3,\infty)$ such that they together fulfill the requirement that
\begin{align*}
    t\Psi(t)\leq \Psi(tK_\Psi)
\end{align*}
for all $t\in [1,\infty)$. Furthermore, we assume that for every $m,n\in\N$ as $\beta m<n<m$ and for every $z\in 3^n\Z^d\cap\square_m$ that the family $\{X_z\}$ of random variables satisfies the following four conditions
\begin{align*}
    \begin{cases}
\mathbb{E}[X_z]=0, \\
|X_z|\leq 1, \\
|D_{z+\square_n}X_z|\leq 1, \\
\textnormal{$X_z$ is $\mathcal{F}(z+\square_n)$-measurable}.
    \end{cases}
\end{align*}
Then, we have the bound for every $t\in[1,\infty)$ that
\begin{align}\label{P3cond}
    \mathbb{P}\left[\left|\avsum_{z\in 3^n\Z^d\cap\square_m} X_z \right|\geq t3^{-\nu(m-n)} \right]\leq \frac{1}{\Psi(t)}.
\end{align}
\end{itemize}

Let us then point out a few remarks regarding the stated assumptions~(P1)--(P3) above. Condition (P1) is a canonical assumption at this point in the contemporary homogenization literature that does not necessarily involve periodic coefficient fields. Traditionally, homogenization theory has studied periodic fields that satisfy $\a(x+z)=\a(x)$ for every $z\in\Z^d$. This theory has been generalized for $\Z^d$-stationary fields over the last few years (see, for example, \cite{ak}), where the assumption of stationarity extends the periodic setting by requiring invariance of the law under translations rather than pointwise periodicity. This, in turn, affects the averaging process within homogenization. The condition of (P1) has been the starting point for the development of this periodicity-free theory (of course, one needs other assumptions as well, such as for governing the ergodicity).

Assumption (P2) provides a weaker and more general sense of ellipticity for the equation in~(\ref{ellipeq}) than the classical uniform ellipticity conditions of (\ref{unifellipconds}) due to the smoothening effect of~$\gamma'$ within smaller scales. Namely, the exponent $\gamma'$ encodes a renormalization-type smoothing, since the ellipticity may deteriorate at small scales, but the coarse-graining restores uniform control at larger scales. Actually, as shown in \cite{AK.HC}, these uniform ellipticity conditions are a special case given by~(P2) after setting $\mathcal{S}=0$ and $\gamma'=0$. However, most importantly, condition~(P2) is \emph{renormalizable}, which is absolutely vital for high-contrast homogenization, since we need to constantly change the geometry under this renormalization to make sense of our objects. Renormalization means in this context that the pushforward of the probability measure $\P$ adheres to the same ellipticity bound under the triadic dilation map $\a\mapsto \a(3^n\cdot)$ for each $n\in\N$. This fact is properly justified in \cite[Proposition 2.11]{AK.HC}. 

The CFS mixing condition in (P3) was originally introduced in \cite{ak} to obtain a convenient and reliable way to quantify the involved ergodicity, as well as to prove optimal quantitative homogenization estimates. Essentially, it is a linear concentration inequality for averaged sums of random variables that depend locally on the coefficient field $\mathbf{a}$. On the one hand, this mixing condition is general enough to incorporate all the random variables that are typically encountered in elliptic homogenization problems. But still, on the other hand, it has to be strong enough to allow optimal quantitative estimates.

Altogether, the assumptions (P1)--(P3) provide us with a very general but still flexible setting for high-contrast homogenization to work with. For example, as rigorously shown in \cite{AK.HC}, many Poisson inclusions, Gaussian stream matrices, and log-normal fields satisfy these conditions.

Now that we have formulated the axiomatic assumptions (P1)--(P3), we should show that they actually imply the condition~\eqref{renormellipt} that is essential in this paper. As mentioned before, the renormalization scheme points our interest towards~\cite[Corollary 4.3]{AK.HC}. For completeness, we will present these (somewhat modified) details below and define some of the objects of Chapter 1 more thoroughly.

However, before formulating the following proposition, let us briefly explain the utilized notation of~$\mathcal{O}_\Psi(C)$ in (\ref{prop21bigoh}). Let $\Psi\colon \R_+\to[1,\infty)$ be an increasing function with the assumption that
\begin{align*}
    \lim_{t\to\infty} \frac{\Psi(t)}{t}=\infty.
\end{align*}
For any non-negative constant $C\geq0$ and random variable $X$, we formalize the following notation $X\leq \mathcal{O}_\Psi(C)$ to mean that
\begin{align*}
    \mathbb{P}[X>tC]\leq \frac{1}{\Psi(t)}
\end{align*}
for each $t\in[1,\infty)$. Furthermore, the notation $X=\mathcal{O}_\Psi(C)$ refers to the case that $|X|\leq \mathcal{O}_\Psi(C)$. For further information on these related quantities of weak Orlicz quasi-norms and their concentration inequalities, see \cite[Appendix C]{AK.HC}.

Lastly, we will introduce the \emph{intrinsic ellipticity ratio} $\Theta$ and the \emph{aspect ratio} $\Pi$ from \cite{AK.HC}. Let us start by decomposing the matrix $\mathbf{E}_0$ from (P2) to the following block form by writing that
\begin{align*}
    \mathbf{E}_0:=\begin{bmatrix}
        \s_0+\k_0^t\s_{*,0}^{-1}\k_0 & -\k_0^t\s_{*,0}^{-1} \\ -\s_{*,0}^{-1}\k_0 & \s_{*,0}^{-1}
    \end{bmatrix}.
\end{align*}
Then, we will define the intrinsic ellipticity ratio $\Theta$ by
\begin{align*}
    \Theta:=\min_{\h\in\R_\textnormal{anti}^{d\times d}}\left| \s_{*,0}^{-\nf 12}(\s_0+(\k_0-\h)^t\s_{*,0}^{-1}(\k_0-\h))\s_{*,0}^{-\nf 12} \right|.
\end{align*}
The aspect ratio $\Pi$ is, in turn, given by $\Pi:=\nf {\Lambda_0}{\lambda_0}$, where we have for the ellipticity constants $0< \lambda_0\leq \Lambda_0< \infty$ that
\begin{align*}
    \lambda_0:=|\s_{*,0}^{-1}|^{-1} \qand \Lambda_0:=\min_{\h\in\R_\textnormal{anti}^{d\times d}}\left| \s_0+(\k_0-\h)^t\s_{*,0}^{-1}(\k_0-\h) \right|.
\end{align*}

\begin{proposition} 
\label{p.Psimplyrenormellipt}
Assume that (P1)--(P3) are valid. First of all, there exists a homogenized matrix~$\bfAhom$ as in~\eqref{defbigahom} and a positive symmetric matrix~$\mathbf{q}_0$ that defines an adapted universal geometry as in~\eqref{e.adapted}. Moreover, for every~$\delta \in (0,1)$ and~$\gamma \in (\gamma',1)$, there exists a constant~$\Upsilon<\infty$ that depends on $d$, $\nu$, $\beta$, $\gamma$, $\gamma'$, $\delta$, $K_\Psi$, $K_{\Psi_\mathcal{S}}$, as well as on the ratios $\Pi$ and $\Theta$, and a non-negative random variable~$\mathcal{Y}_{\delta,\gamma}$ 
that satisfies
\begin{align}\label{prop21bigoh}
\mathcal{Y}_{\delta,\gamma}^{(\nu-\gamma')(1-\beta)} \leq \O_{\Psi}(\Upsilon)
\end{align}
for $\nf{d}{2}\geq\nu>\gamma'$ as follows. Namely, by defining~$\X_{\mathcal{H}} := \max\{ \mathcal{Y}_{\delta,\gamma}, \mathcal{S} \}$, we have 
for every~$n,m \in \Z$ satisfying~$3^m \geq \X_{\mathcal{H}}$ and~$n \leq m$, as well as~$z \in 3^n\mathbb{L}_0 \cap \diamondsuit_m$ that
\begin{align}\label{renormellipt.again}
\mathbf{A}(z+\diamondsuit_n)\leq \biggl(1+3^{\gamma(m-n)}\Bigl(\frac{\mathcal{X}_{\mathcal{H}}}{3^m} \Bigr)^{\! \theta} \biggr)\mathbf{\overline{A}}.
\end{align}
In particular, the conditions (P1)--(P3) imply Assumption~\ref{d.renormellipt}. 
\end{proposition}
\begin{proof}
Assume that (P1)--(P3) are valid. Let us first recall the context of~\cite[Theorem 4.1]{AK.HC} and~\cite[Corollary 4.3]{AK.HC}. These results state that there exist constants~$C(d)<\infty$ and~$c(d) \in (0,\nicefrac12]$ such that, if we define the parameters 
\begin{equation*}
\alpha:= \bigl( \min\{ \nu,1\} - \gamma'\bigr)(1-\beta)
\quad \mbox{and} \quad
\kappa := \min\{ c , \nicefrac \alpha3\},
\end{equation*}
then, for each~$\delta \in (0,1)$ and~$\gamma \in (\gamma',1)$, there exists a random variable~$\tilde{\mathcal{Y}}_{\delta,\gamma}$ satisfying
\begin{equation*}
\tilde{\mathcal{Y}}_{\delta,\gamma}^{(\nu-\gamma')(1-\beta)} \leq
\O_{\Psi} \Biggl(
3^{m_*(\nu-\gamma')(1-\beta)}
\exp\biggl( 
\frac{C }{\gamma-\gamma'}
\log\left(\max\{K_\Psi,\Theta,\delta^{-1} \} \right)
\biggr)
\Biggr).
\end{equation*}
Here, we denoted that
\begin{align*}
    m_*:=C\left(\log K_{\Psi_\mathcal{S}}+\frac{1}{\alpha^2}\log\left(\frac{\Pi K_\Psi}{\alpha} \right) \right)\log(1+\Theta),
\end{align*}
whereas the role of $\delta$ above is to act as the homogenization error tolerance parameter. Moreover, by setting 
\begin{align*}
    \theta := \frac14 \min\left\{ \kappa, \frac{(\nu-\gamma')(1-\beta)(\gamma-\gamma')}{2(d+(\nu-\gamma')(1-\beta))}  \right\},
\end{align*}
$\tilde{\delta}$ to be specified later,~$\mathcal{Y}_{\tilde{\delta},\gamma} := \tilde{\mathcal{Y}}_{\delta,\gamma}$, and~$\X_{\mathcal{H}} := \max\{\mathcal{Y}_{\tilde\delta,\gamma}, \S\}$, there exists~$\bfAhom$ as in~\eqref{defbigahom} so that we have for every~$m,n\in \Z$ with~$n\leq m$ and~$3^m \geq \X_{\mathcal{H}}$ as well as~$ z \in 3^n \Z^d \cap \cu_m$ that
\begin{align}
\label{e.final.P2.again}
\bfA(z+\cu_n) 
\leq 
\biggl(\!1 +\tilde\delta 3^{\gamma(m-n)} \Bigl(\frac{\X_{\mathcal{H}}}{3^m}\Bigr)^{\! \theta} \biggr) \bfAhom.
\end{align}
We define the matrix~$\mathbf{q}_0$ as
\begin{equation}
\label{defqzero.proof}
    (\mathbf{q}_0)_{ij}:=3^{-k_0}\left\lceil 3^{k_0} \left|\shom^{-1}\right|^{\nicefrac 12} (\shom^{\nicefrac 12})_{ij} \right\rceil
\end{equation}
for a constant~$k_0(d) \in \N$ chosen large enough so that $3^{k_0}\mathbb{L}_0\subset\Z^d$. This creates an adapted universal geometry, where we can view every subset $U\subset \R^d$ from the perspective of $\mathbf{q}_0$ as in (\ref{e.adapted}).

We will then check the estimate of~\eqref{renormellipt.again} utilizing a similar Whitney-type decomposition as in \cite[Lemma 2.13]{AK.HC}. 
Let~$l := \lceil \log_3 (C(d) \Pi_{\shom} ^{\nicefrac12}) \rceil$ with a large enough constant~$C<\infty$ for which~$z+\diamondsuit_n \subset \square_{m+l}$ and $\Pi_{\shom}$ defined in \eqref{deflambdabars}. 
Then~$z+\diamondsuit_n$ can be decomposed into a disjoint union (modulo a Lebesgue null set) by the family~$\{ V_j(z) \,|\, -\infty < j \leq n \}$ consisting of such sets that each~$V_j(z)$ is the disjoint union of the cubes~$y+\cu_j$ with~$y \in 3^j\Zd$ as well as 
\begin{equation}\label{prop21pfmeasbds}
|V_j(z)| \leq C\Pi_{\shom}^{\nicefrac12} 3^{j-n} |\diamondsuit_n| 
\quad \mbox{and} \quad 
\sum_{j = -\infty}^n \frac{|V_j(z)|}{|\diamondsuit_n|} = 1.
\end{equation}
We may construct a recursive partition like this in the following manner, for example. We start by defining that
\begin{equation*}
V_n(z):= \bigcup \bigl\{ y+\cu_n \,|\, y \in 3^n \Zd, \; y+\cu_n \subseteq z+\diamondsuit_n \bigr\}.
\end{equation*}
Then, after given the sets~$V_n(z),\ldots,V_j(z)$, we define~$V_{j-1}(z)$ by setting that
\begin{equation*}
V_{j-1}(z):= \bigcup 
\bigl\{
y+ \cu_{j-1} \,|\,
y\in 3^{j-1}\Zd,\, 
y+ \cu_{j-1} \subseteq 
(z+\!\diamondsuit_n) \setminus ( V_{n}(z)  \cup \ldots \cup V_{j}(z) ) 
\bigr \}.
\end{equation*}
Now, we note from \cite[Lemma 5.2]{ak} that the coarse-grained matrices $\mathbf{A}(z+\diamondsuit_n)$ are subadditive with respect to disjoint partitions in the sense that 
\begin{equation}\label{subaddproperty}
    \bfA(z+\diamondsuit_n) 
\leq 
\sum_{j=-\infty}^n
\frac{|V_j(z)|}{|\diamondsuit_n|}
\avsum_{y \in 3^j \Zd \cap V_j(z)}
\bfA(y + \cu_j). 
\end{equation}
Consequently, \eqref{e.final.P2.again} and \eqref{prop21pfmeasbds} provide that
\begin{align*}
\bfA(z+\diamondsuit_n) 
&
\leq 
\biggl(1+\tilde\delta 3^{(\gamma-\theta)l}  
\sum_{j=-\infty}^n
\frac{|V_j(z)|}{|\diamondsuit_n|}
3^{\gamma(m-j)} \Bigl(\frac{\X_{\mathcal{H}}}{3^m}\Bigr)^{\! \theta} 
\biggr) \bfAhom  
\notag \\ &
\leq 
 \biggl( 1 +\tilde\delta 3^{(\gamma-\theta)l} \sum_{j=-\infty}^{n}
C\Pi_{\shom} ^{\nicefrac12}3^{j-n}
3^{\gamma(m-j)} \Bigl(\frac{\X_{\mathcal{H}}}{3^m}\Bigr)^{\! \theta} 
\biggr) \bfAhom  
\notag \\ &
\leq 
\biggl( 1 + \tilde\delta 3^{(\gamma-\theta)l} 3^{\gamma(m-n)} \frac{C\Pi_{\shom}^{\nf 12} }{1-\gamma} \Bigl(\frac{\X_{\mathcal{H}}}{3^m}\Bigr)^{\! \theta} 
\biggr) \bfAhom.
\end{align*} 
By choosing~$\tilde\delta:=  \left(\frac{C3^{(\gamma-\theta)l}\Pi_{\shom}^{\nf 12} }{1-\gamma}\right)^{\!-1}$, we obtain the desired claim of~\eqref{renormellipt.again}. The proof is complete. 
\end{proof}

\subsection*{The adapted geometry}

The most fundamental basic premise in homogenization theory is that for a scale that is large enough, the equation of (\ref{ellipeq}) should \emph{homogenize} to some deterministic equation of the form
\begin{align}\label{homeq}
    -\nabla\cdot \overline{\mathbf{a}}\nabla u=0.
\end{align}
Above, we call the deterministic matrix $\overline{\mathbf{a}}$ the \emph{homogenized matrix}. In other words, the solutions of the equation (\ref{ellipeq}) should converge in sufficiently large scales to the solutions of the homogenized equation (\ref{homeq}). The high-contrast setting causes certain complications in the analysis conducted on our basic objects because otherwise some of them would not be properly defined within the changed geometry. For example, the gradient function $\nabla u$ in \eqref{ellipeq} might not be controllable or even well-defined unless adapted properly by $\s^{1/2}$, as in \eqref{introresult1}. Another valid reason for utilizing the adapted geometry like this is that our estimates would necessarily contain some dependencies on the ellipticity constants. That would be quite problematic for deriving quantitative estimates, since their ratio can be arbitrarily large here.

Naturally, we wish to avoid all of this, and \cite{AK.HC} tackles this problem by introducing the \emph{adapted cubes} $\diamondsuit_n$ while having an ongoing renormalization argument in the background. Our approach in this paper also utilizes this adapted geometry. Hence, we need to redefine the usual open
balls in order to respond to the needs emerging from the convoluted high-contrast framework. Namely, if we denote the expectation of $\mathbf{s}(\square_n)$ by $\overline{\mathbf{s}}(\square_n):=\mathbb{E}(\mathbf{s}(\square_n))$, then we define for every $r>0$ and $y\in\R^d$ that
\begin{align}\label{defadaptedballs}
    B_r:=\{x\in\R^d\; :\; |\mathbf{q}_0^{-1}x|<r \} \qand B_r(y):=\{x\in\R^d\; :\; |\mathbf{q}_0^{-1}(y-x)|<r \}.
\end{align}

Here, essentially $\mathbf{q}_0\approx \overline{\lambda}^{-1/2}\overline{\mathbf{s}}^{1/2}$, which can be seen from our selection of $\mathbf{q}_0$ in the proof of Proposition \ref{p.Psimplyrenormellipt}, where $\overline{\mathbf{s}}$ is the limit of all $\overline{\mathbf{s}}(\square_n)$ as $n\to\infty$ and $\overline{\lambda}$ is the smallest positive eigenvalue of~$\overline{\mathbf{s}}$.
We defined the quantity $\mathbf{q}_0$ in~(\ref{defqzero.proof}) more accurately, but most of the time this approximate description suffices for our demands. To put it simply, (almost) all of the geometry from here on will always be within this adapted geometry. Note that we are handling this new geometry mostly via the aforementioned adapted cubes, but the adapted balls will also be useful for us in Chapter~4, for example.

With this in mind, we recall, for every $k\in\N$, the following subspaces $\mathcal{A}_k$ and $\overline{\mathcal{A}}_k$ of $\mathcal{A}(\R^d)$ and~$\overline{\mathcal{A}}(\R^d)$, respectively. Namely, with this new notation, we set (for adapted open balls this time) that
\begin{equation*}
    \mathcal{A}_k:=\left\{u\in \mathcal{A}(\Rd)\; | \;  \limsup_{r\to\infty} r^{-(k+1)}\|u\|_{\underline{L}^2(B_r)}=0 \right\}
\end{equation*}
and
\begin{equation*}
    \overline{\mathcal{A}}_k:=\left\{u\in \overline{\mathcal{A}}(\Rd)\; | \;  \limsup_{r\to\infty} r^{-(k+1)}\|u\|_{\underline{L}^2(B_r)}=0 \right\}.
\end{equation*}

Lastly, let us recall the \emph{adapted lattice} $\mathbf{q}_0(\Z^d)$ denoted by $\mathbb{L}_0$, that is, $\mathbb{L}_0:=\{\mathbf{q}_0z\; |\; z\in\Z^d \}$. Note that~$\mathbb{L}_0$ ultimately has the same useful properties as the standard lattice~$\Z^d$. Especially, we have that $3^n\mathbb{L}_0\subset \Z^d$ for large enough $n\in\N$. In fact, we choose~$k_0$ in Proposition~\ref{p.Psimplyrenormellipt} so that $3^{k_0}\mathbb{L}_0\subset \Z^d$.

\subsection*{Some previous results}

This section is reserved for listing all of the useful results from the prior homogenization literature (emerging mainly from \cite{ak}, \cite{AK.HC}, and \cite{akm}) that we will require later. Most of the details in the proofs of these results are generally omitted below, but the respective references are always provided and clearly stated for interested readers. Furthermore, some of the claims are slightly modified in accordance with our specific situation. In these cases, we also present short proof sketches to justify them.

In \cite{AK.HC}, much of the content in Chapters 2--4 is reserved to show that the \emph{homogenization error} $\mathcal{E}_n$ is bounded from above by coarse-grained quantities involving sufficiently large scales. In other words, it is shown that we can force $\mathcal{E}_n$ to be small enough for our quantitative computations. Heuristically speaking, the homogenization error $\mathcal{E}_n$ measures the difference that the solution of equation (\ref{ellipeq}) has compared to the solution of the homogenized equation (\ref{homeq}). Consequently, let us begin by reviewing the discussion presented in \cite[Chapter 5]{AK.HC}, which provides a handy upper bound for the error terms $\mathcal{E}_n$ defined below in \eqref{thm21pfstep5defen}. Essentially, this result follows from Assumption \ref{d.renormellipt} in our case.

Before formulating our version of this result, we will provide a rigorous definition for the homogenization error $\mathcal{E}_n$. Namely, by utilizing the same notation as in \cite{AK.HC}, we define for each fixed~$s\in(0,\nf12)$,~$m,n\in\N$ with $n\leq m$, and~$z \in 3^n\mathbb{L}_0 \cap \diamondsuit_m$ that
\begin{equation}\label{thm21pfstep5defen}
\begin{aligned}
&\mathcal{E}_{s}(z+\diamondsuit_n;\mathbf{a},\ahom)^2 \\ 
&:=\left(1-3^{-2s}\right) \! \! \sum_{j=-\infty}^n \! \!  3^{-2s(n-j)}
\! \! \! \! \! \! \max_{y\in z+3^j\mathbb{L}_0 \cap \diamondsuit_n}\max_{|e|= 1}\frac{1}{2}\Bigl(J(\cdot,\shom^{-\nf12}e,\ahom^t\shom^{-\nf 12}e)+J^*(\cdot,\shom^{-\nf 12}e,\ahom \shom^{-\nf 12}e)\Bigr)(y+\diamondsuit_j).
\end{aligned}
\end{equation}
Here, the variational quantities~$J(U,p,q)$ and~$J^*(U,p,q)$ were defined in~\eqref{defJUpq} and~\eqref{defJStarUpq}, respectively. We may refer to this error quantity simply as $\mathcal{E}_n$ if there is no danger of confusion. The factor of~$1-3^{-2s}$ on the right-hand side of \eqref{thm21pfstep5defen} is simply a normalizing constant that allows the mapping of~$n\mapsto (1-3^{-2s})3^{-2ns}$ to be a probability mass function over $\N$.

Essentially, this choice of the homogenization error is involved in the upper bounds of the results in~\cite[Chapter 5]{AK.HC} that we apply later in this paper. It also coincides with the initial induction step in Step 1 of Theorem \ref{thmhighreg} and is bounded by the right-hand side of Proposition~\ref{corcontrolen}. Especially, by looking at \cite[Chapter 5]{AK.HC}, we can see that the aforementioned variational quantities are nicely bounded by a finite spectral norm, which in turn makes the homogenization error finite as well. More precisely, for every bounded Lipschitz domain $U\subset \R^d$ and $e\in\R^d$ with $|e|\leq 1$, we have that
\begin{equation}\label{varqtysarebdd}
\begin{aligned}
    J(&U,\shom^{-\nf12}e,\ahom^t\shom^{-\nf 12}e)+J^*(U,\shom^{-\nf 12}e,\ahom \shom^{-\nf 12}e) \\
    &=\frac{1}{2}\begin{bmatrix}
        -e \\ e
    \end{bmatrix}\cdot \bigl(\bfAhom^{-\nicefrac12}
\bfA(U)\bfAhom^{-\nicefrac12}-I_{2d} \bigr)\begin{bmatrix}
        -e \\ e
    \end{bmatrix}+\frac{1}{2}\begin{bmatrix}
        e \\ e
    \end{bmatrix}\cdot \bigl(\bfAhom^{-\nicefrac12}
\bfA(U)\bfAhom^{-\nicefrac12}-I_{2d} \bigr)\begin{bmatrix}
        e \\ e
    \end{bmatrix} \\
    &\leq 2\left| \bigl(\bfAhom^{-\nicefrac12}
\bfA(U)\bfAhom^{-\nicefrac12}-I_{2d} \bigr)_+ \right|,
\end{aligned}
\end{equation}
where we assumed that $\khom=0$ as explained in Chapter 1. Above, we also denote the positive part of a symmetric matrix $A$ by $A_+$. In summary, the error terms satisfy all the required aspects that we desire from them.

\begin{proposition}[Controlling the homogenization error]\label{corcontrolen}
Under Assumption \ref{d.renormellipt}, we have for every~$s\in(\nf \gamma2,\nf12)$ and $m \in \N$ with~$3^m\geq \mathcal{X_H}$ that there exists a constant $C(d)<\infty$ such that    
\begin{equation}
\label{e.mathcalE.bounds}
\mathcal{E}_{s}(\diamondsuit_m;\mathbf{a},\ahom)
\leq 
\left(\frac{C}{2s - \gamma}    \left(\frac{\mathcal{X_H}}{3^m} \right)^{\theta}\right)^{\nf12}.
\end{equation}
\end{proposition}
\begin{proof}
Fix $m\in\N$ and $s \in (\nf \gamma 2,\nf12)$. 
By~\eqref{renormellipt} and \eqref{varqtysarebdd}, we obtain that
\begin{align*}
\mathcal{E}_{s}(\diamondsuit_m;\mathbf{a},\ahom)^2
&
\leq 
\sum_{n=-\infty}^m 3^{-2s(m-n)}\max_{y\in 3^n\mathbb{L}_0\cap\diamondsuit_m} \Bigl| \bigl(\bfAhom^{-\nicefrac12}
(\bfA(y+\diamondsuit_n)-\bfAhom)
\bfAhom^{-\nicefrac12} \bigr)_+ \Bigr|
\notag \\ &
\leq
\left(\frac{\mathcal{X_H}}{3^m} \right)^{\theta}
\sum_{n=-\infty}^m 3^{-2s(m-n)}  3^{\gamma(m-n)}
\leq
\frac{C}{2s - \gamma} 
\left(\frac{\mathcal{X_H}}{3^m} \right)^{\theta}.
\end{align*}
This completes the proof by taking the square roots from both sides.
\end{proof}

The role of the random variable $\mathcal{X_H}$ is to act as the minimal scale in which homogenization occurs.
More specifically, it is the smallest scale in which the relative homogenization error is smaller than the \emph{error tolerance}~$\delta\in(0,1]$ introduced in \cite{AK.HC} that we also briefly considered in Proposition \ref{p.Psimplyrenormellipt}. Technically, some of our estimates should contain this important parameter~$\delta$, but we chose to omit it, since it does not play a significant role in our analysis. To be more precise, the role of~$\delta$ is important in the larger picture when considering homogenization, as it determines the upper bound for the size of the difference between $u$ in \eqref{ellipeq} and the homogenized solution~$\overline{u}$. However, from our perspective, this would simply mean carrying one extra factor around, which does not play any role in practice. For this reason, we have removed this parameter from our considerations to simplify the expressions.

Next, we wish to establish a result to control the negative Sobolev seminorms in terms of $L^2$ norms. This leads us to the following slightly modified result presented in \cite[Lemma 5.5]{AK.HC} that utilizes Proposition \ref{corcontrolen}. Although it is true that $\nf{\mathcal{X_H}}{3^n}\leq 1$ below, we want to keep this factor included on the right-hand side, as it keeps track of the quantitative decay of the homogenization error that is essential for obtaining quantitative results.

\begin{lemma}[Controlling the Sobolev seminorms]\label{lemadaptpoinc}
    Let $n\in\N$, $s\in(\nf\gamma2,\nf12)$, and~$u\in\mathcal{A}(\diamondsuit_n)$ with $3^n\geq\mathcal{X_H}$. Then, there exists a constant $C(d,s,\gamma)<\infty$ such that
    \begin{align}\label{lem23res1}
3^{-ns}\left[\overline{\mathbf{s}}^{1/2}\nabla u \right]_{\Hminusul(\diamondsuit_n)}+3^{-ns}\left[\overline{\mathbf{s}}^{-1/2}\mathbf{a}\nabla u \right]_{\Hminusul(\diamondsuit_n)}\leq C\left(1+C\left(\frac{\mathcal{X_H}}{3^n} \right)^{\nf{\theta}{2}}\right)\|\mathbf{s}^{1/2}\nabla u \|_{\underline{L}^2(\diamondsuit_n)}
    \end{align}
    and
    \begin{align*}
3^{-ns} \left[\overline{\mathbf{s}}^{-1/2}(\mathbf{a}\nabla u - \ahom \nabla u) \right]_{\Hminusul(\diamondsuit_n)}
\leq 
C\left(\frac{\mathcal{X_H}}{3^n} \right)^{\nf{\theta}{2}}\|\mathbf{s}^{1/2}\nabla u \|_{\underline{L}^2(\diamondsuit_n)}.
    \end{align*}
\end{lemma}
\begin{proof}
    Recall first that the negative Sobolev seminorm is bounded from above by the full negative Sobolev norm. Furthermore, we will denote the symmetric part and the antisymmetric part of an arbitrary matrix $\a_1\in\R_+^{d\times d}$ by $\s_1$ and~$\k_1$, respectively. Then \cite[Lemma 5.5]{AK.HC} provides for all functions $v$ solving the elliptic equation of \eqref{ellipeq} for $\widetilde{\a}\in\Omega$ in $\square_0$ that
    \begin{align*}
        \|\s_1^{1/2}\nabla v \|_{\Hminusul(\square_0)}+\|\s_1^{-1/2}(\widetilde{\mathbf{a}}-\k_1)\nabla v \|_{\Hminusul(\square_0)}\leq C\left(1+\mathcal{E}_s(\square_0;\widetilde{\a},\a_1)\right)\|\widetilde{\mathbf{s}}^{1/2}\nabla v \|_{\underline{L}^2(\square_0)}
    \end{align*}
    and
    \begin{align*}
        \|\s_1^{-1/2}(\widetilde{\mathbf{a}}\nabla v - \a_1 \nabla v) \|_{\Hminusul(\square_0)}\leq 
C\mathcal{E}_s(\square_0;\widetilde{\a},\a_1)\|\widetilde{\mathbf{s}}^{1/2}\nabla v \|_{\underline{L}^2(\square_0)}.
    \end{align*}
    Here, we have for the homogenization errors within standard triadic cubes $\square_n$ that
\begin{equation*}
\begin{aligned}
&\mathcal{E}_{s}(\square_n;\widetilde{\a},\a_1)^2 \\ 
&:=\left(1-3^{-2s}\right) \! \! \sum_{j=-\infty}^n \! \!  3^{-2s(n-j)}
\! \! \! \! \! \! \max_{y\in 3^j\Z^d \cap \square_n}\max_{|e|= 1}\frac{1}{2}\Bigl(J(\cdot,\shom^{-\nf12}e,\ahom^t\shom^{-\nf 12}e)+J^*(\cdot,\shom^{-\nf 12}e,\ahom \shom^{-\nf 12}e)\Bigr)(y+\square_j)
\end{aligned}
\end{equation*}
    for every $n$ and $s$. By making a change of variables for all $x\mapsto \mathbf{q}_0x$, as we recall from earlier that $\diamondsuit_n=\mathbf{q}_0(\square_n)=\{ x\in\R^d\; |\; \mathbf{q}_0^{-1}x\in\square_n \}$, we then have that
    \begin{align*}
    \mathcal{E}_s(\diamondsuit_n;\widetilde{\a},\a_1)=\mathcal{E}_s(\square_n; \overline{\lambda}^{-1}\mathbf{q}_0^{-1}\widetilde{\a}\mathbf{q}_0^{-1}, \overline{\lambda}^{-1}\mathbf{q}_0^{-1}\a_1\mathbf{q}_0^{-1}).
    \end{align*}
    We can also set that $\khom=\k_1=0$ as established in Chapter 1. Then, we will perform a change of variables so that
    \begin{align*}
        u(x):=v(3^{-n}\mathbf{q}_0^{-1}x), \quad \a(x):=\overline{\lambda}\mathbf{q}_0\widetilde{\a}(3^{-n}\mathbf{q}_0^{-1}x)\mathbf{q}_0, \qand \shom:=\overline{\lambda}\mathbf{q}_0\a_1\mathbf{q}_0
    \end{align*}
    and enlarge the domain of the norms from $\square_0$ to $\diamondsuit_n$. Consequently, by noting that $\mathcal{E}_s(\square_0)\leq C\mathcal{E}_s(\diamondsuit_n)$ after a harmless change of variables, Proposition \ref{corcontrolen} yields the claim.
\end{proof}

The concept of \emph{harmonic approximation} is very important in the theory of stochastic homogenization, especially from the viewpoint of regularity theory. Namely, for every solution $u$ of the basic elliptic equation in (\ref{ellipeq}), there exists an $\overline{\mathbf{a}}$-harmonic function that approximates $u$ rather closely. Let us then state a high-contrast version of this fact below that follows quite closely to \cite[Proposition 5.3]{AK.HC}. It should also be noted that the proof of this fact implies the overall homogenization of the Dirichlet problem in high contrast.

\begin{proposition}[Harmonic approximation in high contrast]\label{propharmappr}
    Suppose that $m\in\N$ so that $3^m\geq \mathcal{X_H}$ and $s\in(\nf\gamma2,\nf12)$. Then, for each $u\in\mathcal{A}(\diamondsuit_m)$, there exist an $\overline{\mathbf{a}}$-harmonic function $\overline{u}$ in $\diamondsuit_{m}$ with $u-\overline{u}\in H_c^{1-s}(\diamondsuit_m)$ and a positive constant $C(d,s,\gamma)<\infty$ such that
    \begin{align*}
3^{-m}\overline{\lambda}^{1/2}\|u-\overline{u}\|_{\underline{L}^2(\diamondsuit_{m})}+3^{-ms}\left[\overline{\mathbf{A}}^{1/2}\begin{bmatrix}
       \nabla u- \nabla \overline{u} \\ \mathbf{a}\nabla u -\overline{\mathbf{a}}\nabla\overline{u}
    \end{bmatrix} \right]_{\Hminusul(\diamondsuit_{m})}\leq C\left(\frac{\mathcal{X_H}}{3^m} \right)^{\nf{\theta}{2}}\|\mathbf{s}^{1/2}\nabla u \|_{\underline{L}^2(\diamondsuit_m)}.
    \end{align*}
    Conversely, for each $\overline{\mathbf{a}}$-harmonic function $\overline{u}\in H^{1-s}(\diamondsuit_m)$ in $\diamondsuit_{m}$, there exist $u\in\mathcal{A}(\diamondsuit_m)$ with $u-\overline{u}\in H_c^{1-s}(\diamondsuit_m)$ and a positive constant $C(d,s,\gamma)<\infty$ such that
    \begin{equation*}
    \begin{aligned}
3^{-m}\overline{\lambda}^{\nf12} \|u-\overline{u}\|_{\underline{L}^2(\diamondsuit_{m})}&+3^{-ms}\left[\overline{\mathbf{A}}^{1/2}\begin{bmatrix}
       \nabla u- \nabla \overline{u} \\ \mathbf{a}\nabla u -\overline{\mathbf{a}}\nabla\overline{u}
    \end{bmatrix} \right]_{\Hminusul(\diamondsuit_{m})}\leq C\left(\frac{\mathcal{X_H}}{3^m} \right)^{\nf{\theta}{2}}\|\mathbf{s}^{1/2}\nabla u \|_{\underline{L}^2(\diamondsuit_m)}.
    \end{aligned}
    \end{equation*}
\end{proposition}
\begin{proof}
    From \cite[Proposition 5.3]{AK.HC}, we will get both of the directions above. Namely, let us begin by fixing~$s \in (\nf\gamma2,\nf12)$ and~$\a_1 \in \R_+^{d\times d}$ satisfying the uniform ellipticity conditions of \eqref{unifellipconds} for some ellipticity constants $0<\lambda<\Lambda<\infty$ as required by \cite[Proposition 5.3]{AK.HC} stated below. Then, for $v\in H_{\widetilde{\a}}^1(\square_0)$ solving the elliptic equation of \eqref{ellipeq} for $\widetilde{\a}\in\Omega$ in $\square_0$ and $h\in H^{1-s}(\square_0)$ solving the homogenized equation of \eqref{homeq} for $\a_1$ in $\square_0$ so that $v-h\in H_c^{1-s}(\square_0)$, we have that
    \begin{equation}\label{prop24pfHCres1}
\| v - h\|_{\underline{L}^2(\square_{0})}+
       \|\nabla v- \nabla h \|_{\Hminusul(\square_{0})} + \|\widetilde{\mathbf{a}}\nabla v -\a_1\nabla h \|
    _{\Hminusul(\square_{0})}\leq C\mathcal{E}_{s}(\square_{0};\widetilde{\a},\a_1) \bigl\| \widetilde{\s}^{1/2} \nabla v  \bigr\|_{\underline{L}^{2}(\square_{0})}.
\end{equation}
Here, we utilized the fact again that the negative Sobolev seminorm is bounded above by the full negative Sobolev norm. We also note that $\|v-h\|_{L^2}\leq C\|\nabla(v-h)\|_{\Hminusul}$ above. As in the proof of Proposition \ref{lemadaptpoinc}, we have that $\mathcal{E}_s(\diamondsuit_m;\widetilde{\a},\a_1)=\mathcal{E}_s(\square_m; \overline{\lambda}^{-1}\mathbf{q}_0^{-1}\widetilde{\a}\mathbf{q}_0^{-1}, \overline{\lambda}^{-1}\mathbf{q}_0^{-1}\a_1\mathbf{q}_0^{-1})$. We apply Proposition \ref{corcontrolen} after the same changes of variables as in the proof of Proposition \ref{lemadaptpoinc}, that is,
\begin{align*}
        u(x):=v(3^{-m}\mathbf{q}_0^{-1}x), \quad \overline{u}(x):=h(3^{-m}\mathbf{q}_0^{-1}x), \quad \a(x):=\overline{\lambda}\mathbf{q}_0\widetilde{\a}(3^{-m}\mathbf{q}_0^{-1}x)\mathbf{q}_0, \qand \shom:=\overline{\lambda}\mathbf{q}_0\a_1\mathbf{q}_0
    \end{align*}
to complete the proof.
\end{proof}

For the Caccioppoli inequality that we present in Proposition \ref{prop21} of the subsequent chapter, we need the following elementary inequality. We will apply the following iteration estimate there for volume-normalized $L^2$ norms, which is an essential part of the proof. The proof of this assertion is a straightforward computation that can be found in~\cite[Appendix C]{akm}.

\begin{lemma}\label{lem12}
    Let $A,\xi\geq 0$ be non-negative constants and $\rho\colon \left[\nf{1}{2},1 \right)\to [0,\infty)$ be a non-negative function satisfying the condition that
    \begin{align*}
\sup_{t\in\left[\frac{1}{2},1 \right)}(1-t)^\xi \rho(t)<\infty.
    \end{align*}
    Suppose, additionally, that for every $x,y\in\left[\nf{1}{2},1 \right)$ with $y<x$, it holds that
    \begin{align*}
\rho(y)\leq \frac{1}{2}\rho(x)+(x-y)^{-\xi}A.
    \end{align*}
    Then, there exists some constant $C(\xi)<\infty$ that satisfies the estimate $\rho\left(\frac{1}{2} \right)\leq CA$.
\end{lemma}
\begin{proof}
    Ad verbatim from \cite[Lemma C.6]{akm}.
\end{proof}

\section{Caccioppoli inequality in high contrast}

This chapter focuses on the single most important tool we have in homogenization theory for estimating the energy quantities involved, that is, $L^2$ norms for the (adapted) gradients of $u$. Namely, this tool is called the \emph{Caccioppoli inequality}, and it is a very well-established elliptic regularity estimate, at least for elliptic equations of moderate contrast (see, e.g. \cite[Appendix~C]{akm}). In order to develop any higher-order regularity theory for high-contrast homogenization, it is thus imperative to have a version of this result in this context.

Let us briefly consider the classical uniformly elliptic situation with a single ellipticity constant $\Lambda>1$ first as a warm-up. Consequently, the (interior) homogeneous Caccioppoli inequality states that if $u\in H^1(B_r)$ solves the uniformly elliptic PDE $-\nabla\cdot \mathbf{a}\nabla u=0$ in $B_r$ for some $r>0$, then there exists a constant $C(d,\Lambda)<\infty$ so that
\begin{align}\label{unifcaccioppoli}
    \|\nabla u\|_{\underline{L}^2(B_{r/2})}\leq \frac{C}{r}\|u-(u)_{B_r}\|_{\underline{L}^2(B_r)}.
\end{align}
Here $B_r$ is an open ball of the standard Euclidean geometry and the norm $\|\cdot\|_{\underline{L}^2}$ refers to the volume-normalized $L^2$ norm in the uniformly elliptic setting. One key observation here is that the constant $C$ in (\ref{unifcaccioppoli}) is explicitly dependent on the ellipticity constant $\Lambda$.

We will encounter a similar observation in the uniformly elliptic case of two ellipticity constants satisfying the conditions of (\ref{unifellipconds}). There, we added another ellipticity parameter $\lambda\leq \Lambda$ and studied their moderate ratio $\Pi=\Lambda/\lambda$. Namely, we can write the aforementioned estimate of (\ref{unifcaccioppoli}) in this situation distributionally with smooth test functions $\psi\in C_c^\infty(B_r)$ so that
\begin{align*}
    \|\psi \nabla u\|_{\underline{L}^2(B_r)}^2\leq 4\Pi \|u\nabla \psi\|_{\underline{L}^2(B_r)}^2.
\end{align*}

The main problem with these simple and immensely useful estimates is that we cannot have a similar explicit dependence on the ellipticity ratio $\Pi$, as it could be arbitrarily large, in a high-contrast setting, since it would ruin our quantitative estimates. We would also not be able to utilize the results of \cite{AK.HC} efficiently to build the high-order regularity theory in that case. As mentioned before, to establish any kind of regularity theory for the solutions of (\ref{ellipeq}) in high contrast, we must have some Caccioppoli-type inequality that accommodates these demands in the high-contrast framework. The previous arXiv version of the paper \cite{AK.HC} portrayed two iterative versions of such estimates, one of which we will also state (in somewhat modified form) below to provide sufficient context for the reader. It is also utilized in the proof of Proposition \ref{prop21} (see \cite[Lemma 6.2]{AKv2} for more information).

\begin{lemma}[Iterative Caccioppoli inequality in high contrast]\label{lem21}
    Let $m\in\N$ and $n\in\Z$ with $n<m-2$ and $3^n\geq \mathcal{X_H}$. Then, for each $u\in\mathcal{A}(\diamondsuit_m)$, there is some constant $C(d,s,\gamma)<\infty$ satisfying the estimate
    \begin{align}\label{lem21result1}
\|\mathbf{s}^{1/2}\nabla u\|_{\underline{L}^2(\diamondsuit_{m-1})}\leq C(d)\left(1+C\left(\frac{\mathcal{X_H}}{3^n} \right)^{\nf{\theta}{2}}\right)\left(\overline{\lambda}^{1/2}3^{-m}\|u \|_{\underline{L}^2(\diamondsuit_{m})}+3^{-(m-n)}\|\mathbf{s}^{1/2}\nabla u \|_{\underline{L}^2(\diamondsuit_{m})}\right).
    \end{align}
    \begin{proof}
        Let us first point out the role of the parameter $n$ in the statement. Although we could simply select $n=m-3$, we wish to keep track of this parameter in terms of Proposition \ref{prop21} below. The claim itself follows trivially from \cite[Lemma 6.2]{AKv2} and Proposition~\ref{corcontrolen}. Namely, the statement of \cite[Lemma 6.2]{AKv2} is exactly the same as that of Lemma \ref{lem21}, but the error term of $1+C(\nf{\mathcal{X_H}}{3^m})^{\nf{\theta}{2}}$ is written there optimally by
        \begin{align}\label{lem21optimalerror}
            1+ \max_{y\in3^n\mathbb{L}_0\cap\diamondsuit_m}\frac{1}{2} \! \! \sum_{k=-\infty}^n \! \!  3^{\frac{1}{2}(k-n)}
\! \! \! \! \! \! \max_{z\in y+3^k\mathbb{L}_0 \cap \diamondsuit_n}\sup_{|e|\leq 1}\Bigl(J(\cdot,\shom^{-\nf12}e,\ahom^t\shom^{-\nf 12}e) +J^*(\cdot,\shom^{-\nf 12}e,\ahom \shom^{-\nf 12}e)\Bigr)^{\nf{1}{2}}(z+\diamondsuit_k),
        \end{align}
        whose second term satisfies the same upper bound of Proposition \ref{corcontrolen} as $\mathcal{E}_s(\diamondsuit_n;\a,\ahom)$. Thus, we can utilize Proposition \ref{corcontrolen} to deduce the desired claim.
    \end{proof}
\end{lemma}

With this Caccioppoli estimate, we can already consider the regularity aspects of the solutions. However, when comparing Lemma \ref{lem21} with the uniformly elliptic Caccioppoli inequality in (\ref{unifcaccioppoli}), we immediately notice that Lemma \ref{lem21} is not quite as elegant or simple as we would like it to be. That is, the right-hand side always contains the norm of the adapted gradient $\mathbf{s}^{1/2}\nabla u$ that is already present on the left-hand side as well. This iterative feature is very inconvenient for us because we would ideally like to have only the norm of~$u$ on the right-hand side. Consequently, this requires us to iterate with respect to $m$ while causing many unwanted complications to our computations by adding further summations and such. With higher-order regularity results, computations become even more complicated and robust, while having to work with an additional iteration argument throughout the process. For these reasons, it would indeed be extremely pleasant if such a non-iterative version of the Caccioppoli inequality existed.

The latest version of \cite{AK.HC} fixes this problem by presenting a non-iterative result described above. While being mathematically optimal in the general case, its right-hand side remains rather cumbersome with all of the ellipticity constants involved. That is why we can present the following simpler spiritual successor for it in our specific setting of Assumption \ref{d.renormellipt}. Its proof is also independent of the argument utilized in \cite[Proposition 2.5]{AK.HC}. The only similarity in the proof below is essentially the usage of the same iteration result of Lemma \ref{lem12}. Thus, the result showcased below is not a trivial consequence of their result and has merit of its own.

\begin{proposition}[Non-iterative Caccioppoli inequality in high contrast]\label{prop21}
    Assume that $\frac{1}{2}\leq r < 1$ and $3^M\geq \mathcal{X}$, for some $M\in\N$, with every $z\in 3^k\mathbb{L}_0\cap\diamondsuit_M$ satisfying
    \begin{align}\label{prop21assumption}
\mathbf{A}(z+\diamondsuit_k)\leq \left(1+ 3^{\gamma(M-k)}\left(\frac{\mathcal{X}}{3^M} \right)^\theta \right)\mathbf{\overline{A}}
    \end{align}
    as $k\in\Z\cap(-\infty,M]$. Suppose, furthermore, that
    \begin{align*}
\|\mathbf{s}^{1/2}\nabla u \|_{\underline{L}^2(r\diamondsuit_M)}<\infty
    \end{align*}
    for $u\in\mathcal{A}(\diamondsuit_M)$. Then, there exist an exponent $\kappa(d,\gamma)>0$ and a constant $C(\kappa)<\infty$ for which it holds that
    \begin{align}\label{prop21claim}
\|\mathbf{s}^{1/2}\nabla u \|_{\underline{L}^2(r\diamondsuit_M)}\leq C\overline{\lambda}^{1/2}3^{-M}(1-r)^{-\kappa}\|u\|_{\underline{L}^2(\diamondsuit_M)}.
    \end{align}
    \begin{proof}
    \emph{Step 1.} Let us first justify the following assertion. Namely, there exists a constant $C(\kappa)<\infty$ such that
    \begin{align}\label{prop21pfstep1goal}
\|\mathbf{s}^{1/2}\nabla u \|_{\underline{L}^2(r\diamondsuit_M)}\leq C\overline{\lambda}^{1/2}(R-r)^{-\kappa}3^{-M}\|u\|_{\underline{L}^2(\diamondsuit_M)}+\frac{1}{2} \|\mathbf{s}^{1/2}\nabla u \|_{\underline{L}^2(R\diamondsuit_M)}
    \end{align}
    for any fixed $R\in(r,1]$. We choose the largest $m\in\Z$ for which we have that $3^{m+2}\leq (R-r)3^M$. Then we may cover $r\diamondsuit_M$ with a family of cubes $\{z+\diamondsuit_{m}\}$ as $z\in 3^{m}\mathbb{L}_0\cap\diamondsuit_M$ such that $z+\diamondsuit_{m+1}\subset R\diamondsuit_M$ for every $z$. By plugging the assumption of (\ref{prop21assumption}) into the left-hand side below, we may deduce that
    \begin{align}\label{prop21pfbound1}
\left|\left(\overline{\mathbf{A}}^{-1/2}\left(\mathbf{A}(z+\diamondsuit_m)-\overline{\mathbf{A}} \right) \overline{\mathbf{A}}^{-1/2} \right)_+ \right|\leq 3^{\gamma(M-m)}\left(\frac{\mathcal{X}}{3^M} \right)^\theta.
    \end{align}
    Consider now the following alternative error quantity $\widetilde{\mathcal{E}}_s(y+\diamondsuit_n)$ defined for all $s\in(0,1]$, $n\in\N$, and $y\in\R^d$ by
    \begin{align*}
\widetilde{\mathcal{E}}_s(y+\diamondsuit_n):=\sum_{k=-\infty}^n 3^{s(k-n)} \max_{z\in y+3^k\mathbb{L}_0\cap\diamondsuit_n} \left|\left(\overline{\mathbf{A}}^{-1/2}(\mathbf{A}(z+\diamondsuit_k)-\overline{\mathbf{A}})\overline{\mathbf{A}}^{-1/2} \right)_+ \right|^{\nf{1}{2}}.
    \end{align*}
    From the minimal scale assumption of $3^M\geq\mathcal{X}$ and by setting that $n\leq M$, we can change the indexing within the aforementioned maximum so that
    \begin{align*}
\widetilde{\mathcal{E}}_s(y+\diamondsuit_n)\leq \sum_{k=-\infty}^n 3^{s(k-n)} \max_{z\in y+3^k\mathbb{L}_0\cap\diamondsuit_M} \left|\left(\overline{\mathbf{A}}^{-1/2}(\mathbf{A}(z+\diamondsuit_k)-\overline{\mathbf{A}})\overline{\mathbf{A}}^{-1/2} \right)_+ \right|^{\nf{1}{2}}.
    \end{align*}
    Furthermore, by the assumption of (\ref{prop21assumption}) and (\ref{prop21pfbound1}), we may then find a constant $C(d)<\infty$ for the following estimate that
    \begin{align*}
\max_{z\in y+3^k\mathbb{L}_0\cap\diamondsuit_M} \left|\left(\overline{\mathbf{A}}^{-1/2}(\mathbf{A}(z+\diamondsuit_k)-\overline{\mathbf{A}})\overline{\mathbf{A}}^{-1/2} \right)_+ \right|^{\nf{1}{2}}\leq C(d)3^{\frac{\gamma}{2}(M-k)}\left(\frac{\mathcal{X}}{3^M} \right)^{\theta/2}.
    \end{align*}
    This observation allows us to conclude, after some clever indexing, that
    \begin{align}\label{deftildecales}
\widetilde{\mathcal{E}}_s(y+\diamondsuit_n)\leq \sum_{k=0}^\infty 3^{-(s-\gamma/2)k} C3^{\frac{\gamma}{2}(M-n)}\left(\frac{\mathcal{X}}{3^M} \right)^{\theta/2}\leq C(d,s) 3^{\frac{\gamma}{2}(M-n)},
    \end{align}
    where we take $s>\gamma/2$ so that the geometric series in the middle converges. In order to recap our computations so far, we have with the assumptions $n\leq M$ and $y\in3^n\mathbb{L}_0\cap\diamondsuit_M$ that
    \begin{align*}
\widetilde{\mathcal{E}}_s(y+\diamondsuit_n)\leq C 3^{\frac{\gamma}{2}(M-n)}<\infty.
    \end{align*}
    
    We continue by utilizing Lemma \ref{lem21}, or more precisely, the respective formulation of \cite[Lemma 6.2]{AKv2} that has the optimal and scale-invariant error term presented in \eqref{lem21optimalerror}, which is also applicable for negative values of $m\in\Z$. This yields for a constant $C(d)>\infty$ that
    \begin{equation*}
    \begin{aligned}
\|\mathbf{s}^{1/2}\nabla u \|_{\underline{L}^2(z+\diamondsuit_m)}\leq C\left(1+3^{\frac{\gamma}{2}(M-n)} \right)\Big(&\overline{\lambda}^{1/2}3^{-(m+1)}\|u\|_{\underline{L}^2(z+\diamondsuit_{m+1})} 
+3^{-((m+1)-n)}\|\mathbf{s}^{1/2}\nabla u \|_{\underline{L}^2(z+\diamondsuit_{m+1})} \Big)
    \end{aligned}
    \end{equation*}
    after choosing that $s=\nf{1}{2}$ in order to match the homogenization errors above. Note that here $n<m-2$ and $z+\diamondsuit_{m+1}\subset R\diamondsuit_M$. We also utilized the fact that by \cite[Lemma 6.2]{AKv2}, it holds that the second term of \eqref{lem21optimalerror} is bounded above by $2\widetilde{\mathcal{E}}_s(y+\diamondsuit_n)$.
    Consequently, due to our construction, we can now confirm with a simple covering argument that
    \begin{equation*}
    \begin{aligned}
\|\mathbf{s}^{1/2}\nabla u \|_{\underline{L}^2(r\diamondsuit_M)}&\leq C\left(1+3^{\frac{\gamma}{2}(M-n)} \right)\overline{\lambda}^{1/2}3^{-m}\|u\|_{\underline{L}^2(R\diamondsuit_{M})} \\
&+C\left(1+3^{\frac{\gamma}{2}(M-n)} \right) 3^{-(m-n)}\|\mathbf{s}^{1/2}\nabla u \|_{\underline{L}^2(R\diamondsuit_{M})}.
    \end{aligned}
    \end{equation*}
    
    Furthermore, we choose the scale $n$ below so that the difference $m-n\geq 3$ is sufficiently large but finite. Since now $3^{m+2}\leq (R-r)3^M\leq 3^{m+3}$, we have that $3^{-m}\leq C(R-r)^{-1}3^{-M}$ and, thus, there exists a positive constant $C<\infty$ satisfying that
    \begin{align*}
3^{\frac{\gamma}{2}(M-m)}\leq C(R-r)^{-\nf{\gamma}{2}},
    \end{align*}
    where the right-hand side is at least one. Because $\nf{\gamma}{2}-1<0$, there is some small $\varepsilon>0$ after choosing $n$ properly such that
    \begin{align*}
3^{-(m-n)}\left(1+3^{\frac{\gamma}{2}(M-n)}\right)
\leq 3^{-(m-n)}+C(R-r)^{-\nf{\gamma}{2}}3^{(\nf{\gamma}{2}-1)(m-n)} \leq \varepsilon.
    \end{align*}
    Consequently, there exists now some positive exponent $\kappa>0$ by optimization for which
    \begin{align*}
\|\mathbf{s}^{1/2}\nabla u \|_{\underline{L}^2(r\diamondsuit_M)}\leq C(\kappa)\overline{\lambda}^{1/2}(R-r)^{-\kappa}3^{-M}\|u\|_{\underline{L}^2(\diamondsuit_M)}+C\varepsilon \|\mathbf{s}^{1/2}\nabla u \|_{\underline{L}^2(R\diamondsuit_M)}.
    \end{align*}
    We can certainly choose $\varepsilon>0$ to be sufficiently small so that $C\varepsilon\leq \nf{1}{2}$ above, which implies the claim in (\ref{prop21pfstep1goal}).
    
    \emph{Step 2.} Now, we will conclude the proof by utilizing Lemma \ref{lem12}. Namely, we define the function $\rho\colon [\nf{1}{2},1)\to [0,\infty)$ by
    \begin{align*}
\rho(r):=\|\mathbf{s}^{1/2}\nabla u \|_{\underline{L}^2(r\diamondsuit_M)}.
    \end{align*}
    We choose the non-negative quantities such that $\xi:=\kappa>0$ and
    \begin{align*}
A:=C\overline{\lambda}^{1/2}3^{-M}\|u\|_{\underline{L}^2(\diamondsuit_M)}\geq 0,
    \end{align*}
    where the constant $C(\kappa)<\infty$ is the same as in (\ref{prop21pfstep1goal}). Now, it is apparent that
    \begin{align*}
\sup_{t\in[\nf{1}{2},1)}(1-t)^\kappa\rho(t)<\infty.
    \end{align*}
    Furthermore, by rewriting the result of the previous step with this language, we have the estimate that
    \begin{align*}
\rho(r)\leq \frac{1}{2}\rho(R)+(R-r)^{-\xi}A
    \end{align*}
    for every $\nf{1}{2}\leq r<R\leq 1$. Consequently, since all assumptions obviously hold, we can then utilize Lemma \ref{lem12}, which states that there exists a constant $C(\kappa)<\infty$ for which
    \begin{align*}
\rho\left(\frac{1}{2} \right)\leq CA.
    \end{align*}
    The desired claim follows after iterating with respect to $r$.
    \end{proof}
\end{proposition} 

Before proceeding to the final chapter of this article, where we consider the higher-order regularity theory, let us still point out a few complementary remarks regarding the non-iterative Caccioppoli inequality. First of all, it is a reasonable question whether the bound in the assumption of (\ref{prop21assumption}) is sensible at all and how restrictive this requirement is exactly. The answer to this question can be found in \cite[Corollary 4.3]{AK.HC} and it turns out to be a very natural and non-rigid assumption given the underlying presumption (P2). As we already noticed in Proposition \ref{p.Psimplyrenormellipt}, this bound is a direct implication of (P2). The result of Proposition \ref{prop21} implies, in turn, the application shown in Proposition \ref{introCI}, since it always holds that $\diamondsuit_{m-1}\subset r\diamondsuit_m$.

Secondly, when comparing the statement of Proposition \ref{prop21} to its iterative counterpart in Lemma \ref{lem21}, it is apparent that the estimate of (\ref{prop21claim}) is a more convenient and powerful result than the one in (\ref{lem21result1}). This is simply due to the fact that Proposition \ref{prop21} does not contain the (adapted) energy norm on the right-hand side (as all Caccioppoli inequalities should ideally not), whereas Lemma \ref{lem21} does. Consequently, when applying the result of Lemma \ref{lem21}, one needs to iterate these estimates endlessly across greater and greater scales to achieve any sensible quantitative bounds. In contrast to this, we do not need to do any of that with Proposition \ref{prop21}, which provides a direct estimate for the $L^2$ norm of $u$. It is needless to say that this result will help us a lot within this paper to build the high-order regularity theory, and also, most certainly, in the future as well.

It should also be pointed out that a coarse-grained version similar to Proposition \ref{prop21} has already been formulated and proved in \cite[Proposition 5.24]{ak} for the uniformly elliptic regime of which \cite[Proposition 2.5]{AK.HC} is generalized to the high-contrast setting. However, as is customary for the uniformly elliptic homogenization theory, the ellipticity ratio is present on its right-hand side there. For the high-contrast framework, we must be much more careful, since these ratios can be arbitrarily large. Otherwise, these propositions share much in common, at least on the qualitative side. One drawback of this particular high-contrast version of the Caccioppoli inequality is the increased randomness caused by the presence of the random minimal scale $\mathcal{X}$ as the statement of \cite[Proposition 2.5]{AK.HC} does not depend on it, but we can usually live with this trade-off.

Now that we have established our version of the Caccioppoli inequality, we are ready to proceed to the higher-order regularity theory, which relies heavily on this result. Let us conclude this chapter by recording a direct corollary of Proposition \ref{prop21} that utilizes the adapted balls defined in (\ref{defadaptedballs}) instead of the adapted cubes as above. This alternate formulation is of some use to us in the proof of Theorem \ref{thmhighreg}. The proof of this claim is essentially the same as above and, thus, we will be omitting it. The only changes would be the obvious ones by replacing the cubes with the adapted balls and altering the different quantities accordingly.
\begin{corollary}\label{cor21}
    \textnormal{(Non-iterative Caccioppoli inequality for adapted balls)} Assume that $0< \nicefrac{R}{2}\leq r < R$ and $R\geq \mathcal{X}$ with every $z\in 3^k\mathbb{L}_0\cap B_R$ satisfying
    \begin{align*}
\mathbf{A}(z+\diamondsuit_k)\leq \left(1+\left(\frac{R}{3^k} \right)^{\gamma}\left(\frac{\mathcal{X}}{R} \right)^\theta \right)\mathbf{\overline{A}}
    \end{align*}
    as $k\in\Z$ and $3^k\leq R$. Suppose, furthermore, that
    \begin{align*}
\|\mathbf{s}^{1/2}\nabla u \|_{\underline{L}^2(B_r)}<\infty
    \end{align*}
    for $u\in\mathcal{A}(B_R)$. Then, there exist an exponent $\kappa(d,\gamma)>0$ and a constant $C(\kappa)<\infty$ for which it holds that
    \begin{align*}
\|\mathbf{s}^{1/2}\nabla u \|_{\underline{L}^2(B_r)}\leq C\overline{\lambda}^{1/2}R^{-1}(1-r/R)^{-\kappa}\|u\|_{\underline{L}^2(B_R)}.
    \end{align*}
\end{corollary}

\section{High-order regularity theory in high contrast}

In this final chapter of the article, we will consider the theory of high-order regularity for the elliptic PDE of (\ref{ellipeq}) in the setting of high-contrast homogenization. This topic has yet to be studied extensively since it was left out of the paper \cite{AK.HC}. The regularity theory of elliptic PDEs studies the aspects of regularity (such as smoothness or integrability, etc.) of solutions $u$ in the equation~(\ref{ellipeq}). The classical high-order regularity theory for uniformly elliptic equations is well established by now (see \cite{ak} or \cite{akm} for an extensive overview), starting already from the 1980s with the works of \cite{avellaneda} and \cite{lin}. Since then and even before, there have been many classical regularity results, including the De Giorgi-Nash, Meyers, and Schauder estimates, as well as the Calderón-Zygmund estimate and, of course, the Caccioppoli estimate.

The construction of such an extensive theory for high-contrast homogenization has been widely open, whereas some similar regularity results are presented in \cite{AK.HC} and \cite{asmart}. Even more so, the theory of high-order regularity has never been studied before in the high-contrast context of this article. In some sense, we try to generalize or modify the already existing arguments for uniformly elliptic equations to prove the same results in the high-contrast framework. However, this is not as straightforward as it may seem at first glance because these arguments often utilize uniform ellipticity extensively, and thus it often requires completely different arguments and angles of approach.

In the article \cite{AK.HC}, quite a lot of effort was paid to ensure that the quantitative estimates obtained there would satisfy even the worst randomness imaginable. Although this was required for mathematical completeness, it led to very complicated and cumbersome error terms. This was due to the unpredictable setting in which the high-contrast objects live while trying to quantify the error caused by the homogenization process. In the framework of this article, we take a step back from the worst randomness and set some bounds for it in the form of Assumption \ref{d.renormellipt}. This provides us with the means to control these error terms, as documented in Proposition \ref{corcontrolen}. This is, after all, a reasonable computational simplification, which does not really affect the analysis too much.

The following theorem is the main result of this article. Note that the statement (along with its proof as well) follows the overall structure of \cite[Theorem 6.12]{ak} or, alternatively, \cite[Theorem 3.8]{akm}. Of course, there are many modifications caused by the high-contrast setting within the statement and its proof. This is especially apparent within the arguments that we have utilized for Steps 2 and 3 in the proof. In the following theorem, the result is stated for a fixed regularity parameter $k\in\N$ referring to the induction argument used in the proof.

\begin{theorem}[Large-scale $C^{k,1}$ regularity in high contrast]\label{thmhighreg}
    Suppose that Assumption \ref{d.renormellipt} holds. Then, the following assertions hold for every fixed $k\in\N$ and $x\in\diamondsuit_1$.
    \begin{itemize}
\item[(i)$_k$] For every $u\in\mathcal{A}_k$, $s\in(\nf{\gamma}{2},\nf12)$, and $n\in\N$, there exist some $\overline{u}\in\overline{\mathcal{A}}_k$ and a constant $C(d,k,s,\gamma)<\infty$ given that $3^n\geq \mathcal{X_H}$ for which
\begin{align}\label{thmhighregclaim1}
    3^{-n}\overline{\lambda}^{1/2}\|u-\overline{u} \|_{\underline{L}^2(\diamondsuit_n)}+3^{-ns}\left[\overline{\mathbf{A}}^{1/2}\begin{bmatrix}
\nabla u-\nabla \overline{u} \\ \mathbf{a}\nabla u-\ahom\nabla\overline{u}
    \end{bmatrix} \right]_{\Hminusul(\diamondsuit_n)}\leq C
    \left(\frac{\mathcal{X_H}}{3^n} \right)^{\nf{\theta}{2}} \|\overline{\mathbf{s}}^{1/2}\nabla \overline{u} \|_{\underline{L}^2(\diamondsuit_n)}.
\end{align}
\item[(ii)$_k$] For every $\overline{u}\in\overline{\mathcal{A}}_k$, $s\in(\nf{\gamma}{2},\nf12)$, and $n\in\N$, there exist some $u\in\mathcal{A}_k$ and a constant $C(d,k,s,\gamma)<\infty$ given that $3^n\geq \mathcal{X_H}$ for which
\begin{align}\label{thmhighregclaim2}
    3^{-n}\overline{\lambda}^{1/2}\|u-\overline{u} \|_{\underline{L}^2(\diamondsuit_n)}+3^{-ns}\left[\overline{\mathbf{A}}^{1/2}\begin{bmatrix}
\nabla u-\nabla \overline{u} \\ \mathbf{a}\nabla u-\ahom\nabla\overline{u}
    \end{bmatrix} \right]_{\Hminusul(\diamondsuit_n)}\leq C
    \left(\frac{\mathcal{X_H}}{3^n} \right)^{\nf{\theta}{2}} \|\overline{\mathbf{s}}^{1/2}\nabla \overline{u} \|_{\underline{L}^2(\diamondsuit_n)}.
\end{align}
\item[(iii)$_k$] There exists a constant $C(d,k,s,\gamma)<\infty$ such that, for all $R\geq \mathcal{X_H}$ and~$u\in\mathcal{A}(B_R(x))$, the following claim holds. Namely, for each $r\in[\mathcal{X_H},R]$, there exists some $\phi\in\mathcal{A}_k$ so that
\begin{align}\label{thmhighregclaim3}
    \|\mathbf{s}^{1/2}\nabla(u-\phi) \|_{\underline{L}^2(B_r(x))}\leq C\left(\frac{r}{R} \right)^k \|\mathbf{s}^{1/2}\nabla u \|_{\underline{L}^2(B_R(x))}.
\end{align}
    \end{itemize}
    Especially, we have almost surely with respect to $\mathbb{P}$ that for all $k\in\N$, it holds that
    \begin{align}\label{thmhighregdims}
\dim(\mathcal{A}_k)=\dim(\overline{\mathcal{A}}_k)=\binom{d+k-1}{k}+\binom{d+k-2}{k-1}.
    \end{align}
    \begin{proof}
The structure of this proof is rather complicated, so let us begin by describing it first. We proceed with an induction loop on $k\in\N$ for which the initial step $k=0$ is checked in Step 1. For the subsequent steps, we prove the induction step from $k-1$ to $k$ by justifying the following implications
\begin{align*}
    &\textnormal{(i)$_{k-1}$, (ii)$_{k-1}$, and (iii)$_{k-1}^{'}$}\Longrightarrow \textnormal{(ii)$_{k}$}, \\
    &\textnormal{(i)$_{k-1}$ and (ii)$_{k}$}\Longrightarrow \textnormal{(i)$_{k}$}, \\
    &\textnormal{(i)$_{k}$ and (ii)$_{k}$}\Longrightarrow \textnormal{(iii)$_{k}^{'}$}.
\end{align*}
In the above, we include the slightly weaker claim of (iii)$_{k}^{'}$ compared to (\ref{thmhighregclaim3}), which we will formulate next. That is, the statement of (iii)$_k$ remains unchanged otherwise, but the estimate in (\ref{thmhighregclaim3}) is modified to
\begin{align*}
    \|\mathbf{s}^{1/2}\nabla(u-\phi) \|_{\underline{L}^2(B_r(x))}\leq C\left(\frac{r}{R} \right)^{k-\rho} \|\mathbf{s}^{1/2}\nabla u \|_{\underline{L}^2(B_R(x))}
\end{align*}
as $\rho\in(0,\nf{\theta}{2})$. Then, in the fifth step of the proof, we show that the claims of (i)$_{k+1}$, (ii)$_{k+1}$, and~(iii)$_{k+1}^{'}$ imply the upgraded form of (iii)$_{k}$. The final sixth step repeats the argument from \cite[Theorem 3.8]{akm} for (\ref{thmhighregdims}), which also concludes the proof of the entire theorem.

\emph{Step 1. Initial step of $k=0$.} The first thing to note, for all of the steps in the proof, is that the random scale $\mathcal{X_H}$ always exists by Assumption \ref{d.renormellipt} with the desired integrability properties. For~(ii)$_0$, it suffices to notice that $\overline{u}\in\overline{\mathcal{A}}_0$ is $\overline{\mathbf{a}}$-harmonic and therefore, by Liouville's theorem, a constant function as a harmonic polynomial of degree zero. This implies that $\overline{\mathcal{A}}_0\subset \mathcal{A}_0$,
which in turn means that~(ii)$_0$ holds trivially by choosing $u=\overline{u}$. For (i)$_0$, we note that actually, even $\overline{\mathcal{A}}_0 = \mathcal{A}_0$ holds since Proposition \ref{introCI} yields for every $n\in\N$ with $3^n\geq \mathcal{X_H}$ that
\begin{align*}
    \|\s^{1/2}\nabla u\|_{\underline{L}^2(\diamondsuit_n)}\leq C\overline{\lambda}^{1/2}3^{-(n+1)}\|u\|_{\underline{L}^2(\diamondsuit_{n+1})}\leq C\overline{\lambda}^{1/2}\limsup_{m\to\infty}3^{-m}\|u\|_{\underline{L}^2(\diamondsuit_m)}
\end{align*}
for some $m\geq n+1$. Now, the rightmost side above tends to zero, as $u\in\mathcal{A}_0$, which means that $u$ is a constant almost everywhere. This implies that $\overline{\mathcal{A}}_0 = \mathcal{A}_0=\{\textnormal{a.e. constant functions} \}$, and we can conclude (i)$_0$ by choosing yet again that $u=\overline{u}$. Lastly, for (iii)$_0$, we simply utilize the triangle inequality for any constant function $\phi\in\mathcal{A}_0$, which instantly provides the claim. Consequently, we have now verified that the initial step $k=0$ holds in our induction argument.

\emph{Step 2. Proof of (ii)$_k$.} For our induction assumption, we suppose that the assumptions (i)$_{k-1}$, (ii)$_{k-1}$, and (iii)$_{k-1}^{'}$ hold. Assumption \ref{d.renormellipt} provides us with a positive exponent $\theta>0$, which we fix now universally even for~(iii)$_{k}^{'}$. The exponent of $\nf\theta2$ within the claim arises from the right-hand side of Proposition~\ref{corcontrolen}.

From the harmonic approximation result in Proposition \ref{propharmappr}, we have for any harmonic polynomial $\overline{u}\in\overline{\mathcal{A}}_k$ that there exist a function $u_n\in\mathcal{A}(\diamondsuit_n)$ and a constant $C(d,s,\gamma)<\infty$ satisfying
\begin{equation}\label{thm21pfstep2harmappr}
\begin{aligned}
    3^{-n}\overline{\lambda}^{1/2}\|u_n-\overline{u}\|_{\underline{L}^2(\diamondsuit_n)}+3^{-ns}\left[\overline{\mathbf{A}}^{1/2}\begin{bmatrix}
\nabla(u_n-\overline{u}) \\ \mathbf{a}\nabla u_n-\overline{\mathbf{a}}\nabla\overline{u}
    \end{bmatrix}  \right]_{\Hminusul(\diamondsuit_n)} \leq C\left(\frac{\mathcal{X_H}}{3^n} \right)^{\nf{\theta}{2}}\|\mathbf{s}^{1/2}\nabla u_n \|_{\underline{L}^2(\diamondsuit_n)}.
\end{aligned}
\end{equation}
We may utilize the result of \cite[Proposition 5.3]{AK.HC} by setting there that $g=\overline{u}$ on $\partial\diamondsuit_n$ as a harmonic polynomial. This result states (after utilizing Proposition \ref{corcontrolen}, a reabsorption in the second inequality, and norm equivalence) that
\begin{align}\label{fromutouhom}
    C\left(\frac{\mathcal{X_H}}{3^n} \right)^{\nf{\theta}{2}}\|\mathbf{s}^{1/2}\nabla u_n \|_{\underline{L}^2(\diamondsuit_n)}\leq C\left(\frac{\mathcal{X_H}}{3^n} \right)^{\nf{\theta}{2}}\left(1+\mathcal{E}_s(\diamondsuit_n; \a, \ahom) \right)\|\nabla \overline{u}\|_{\underline{H}^s(\diamondsuit_n)} \leq C\left(\frac{\mathcal{X_H}}{3^n} \right)^{\nf{\theta}{2}}\|\shom^{1/2}\nabla\overline{u}\|_{\underline{L}^2(\diamondsuit_n)}.
\end{align}
The same argument as above also justifies the minimum on the right-hand side of Theorem \ref{thm13}.
Consequently, it holds (for a larger constant $C(d,s,\gamma)<\infty$) that
\begin{align}\label{thm21phstep2ptgoal}
    3^{-n}\overline{\lambda}^{1/2}\|u_n-\overline{u}\|_{\underline{L}^2(\diamondsuit_n)}+3^{-ns}\left[\begin{bmatrix}
\overline{\mathbf{s}}^{1/2}\nabla(u_n-\overline{u}) \\ \overline{\mathbf{s}}^{-1/2}(\mathbf{a}\nabla u_n-\overline{\mathbf{a}}\nabla\overline{u})
    \end{bmatrix}  \right]_{\Hminusul(\diamondsuit_n)}\leq C\left(\frac{\mathcal{X_H}}{3^n} \right)^{\nf{\theta}{2}}
    \|\overline{\mathbf{s}}^{1/2}\nabla \overline{u} \|_{\underline{L}^2(\diamondsuit_n)}.
\end{align}
Now, it remains to upgrade the aforementioned functions $u_n$ from $\mathcal{A}(\diamondsuit_n)$ to $\mathcal{A}_k$ by showing that~$w_n$ and~$\phi_n$ below are close enough to each other. Let us thus define the difference function $w_n:=u_{n+1}-u_n\in\mathcal{A}(\diamondsuit_{n})$ and its corresponding corrector $\phi_n\in\mathcal{A}_{k-1}$ that emerges from (iii)$_{k-1}^{'}$.

\emph{Step 2.1. Closeness of $w_n$ and $\phi_n$.} In this step, we will study the closeness of $w_n$ and $\phi_n$ more thoroughly, as explained above. Continuing from before, we wish to utilize the statement of~(iii)$_{k-1}^{'}$, which provides for the scale $3^l\in[\mathcal{X_H},3^{n-1}]$ that 
\begin{align}\label{thm21pfstep2estim1}
    \|\mathbf{s}^{1/2}\nabla(w_n-\phi_n)\|_{\underline{L}^2(\diamondsuit_l)}\leq C3^{-((n-1)-l)(k-1-\rho)}\|\mathbf{s}^{1/2}\nabla w_n\|_{\underline{L}^2(\diamondsuit_{n-1})}.
\end{align}
Then, we can reason with the non-iterative Caccioppoli inequality of Proposition \ref{prop21} that
\begin{align}\label{thm21pfstep21nic}
    \|\mathbf{s}^{1/2}\nabla w_n\|_{\underline{L}^2(\diamondsuit_{n-1})}\leq C3^{-n}\overline{\lambda}^{1/2}\|w_n-(w_n)_{\diamondsuit_n}\|_{\underline{L}^2(\diamondsuit_{n})}.
\end{align}
Here and further along the proof, we denote the integral average of $w_n$ in $\diamondsuit_n$ by $(w_n)_{\diamondsuit_n}$. For the right-hand side above, we recall our definition for $w_n$ and apply the triangle inequality to establish that
\begin{align*}
    \|w_n-(w_n)_{\diamondsuit_{n}}\|_{\underline{L}^2(\diamondsuit_{n})}\leq \|u_n-\overline{u}-(u_n-\overline{u})_{\diamondsuit_n}\|_{\underline{L}^2(\diamondsuit_{n})}+\|u_{n+1}-\overline{u}-(u_{n+1}-\overline{u})_{\diamondsuit_n}\|_{\underline{L}^2(\diamondsuit_{n})}.
\end{align*}
Thus, the estimate in (\ref{thm21phstep2ptgoal}) allows us to conclude that
\begin{align*}
    \|\mathbf{s}^{1/2}\nabla w_n\|_{\underline{L}^2(\diamondsuit_{n-1})}\leq 3^{-n}\overline{\lambda}^{1/2}\|w_n-(w_n)_{\diamondsuit_n}\|_{\underline{L}^2(\diamondsuit_{n})}\leq C\left(\frac{\mathcal{X_H}}{3^n} \right)^{\nf{\theta}{2}}\|\overline{\mathbf{s}}^{1/2}\nabla \overline{u} \|_{\underline{L}^2(\diamondsuit_{n})}.
\end{align*}
The only remaining problem here is that we cannot directly estimate the (rescaled) $L^2$ norm of~$\s^{1/2}\nabla w_n$ in $\diamondsuit_n$ while remaining on the same scale. However, luckily for us, this can be simply tackled by changing the scale and increasing the constants in the end. Namely, starting from (\ref{thm21pfstep21nic}), we have concluded from our reasoning above that
\begin{align}\label{thm21pfstep2crudebd}
    \|\mathbf{s}^{1/2}\nabla w_n\|_{\underline{L}^2(\diamondsuit_{n-1})}
    \leq C\left(\frac{\mathcal{X_H}}{3^n} \right)^{\nf{\theta}{2}}
    \|\mathbf{\overline{s}}^{1/2}\nabla \overline{u}\|_{\underline{L}^2(\diamondsuit_{n})}.
\end{align}

Reverting back to the task at hand, we will consider the following partial summations
\begin{align*}
    v_m:=u_m-\sum_{j=1}^{m-1} \phi_j=u_{n+1}+\sum_{j=n+1}^{m-1} (u_{j+1}-u_j-\phi_j)-\sum_{j=1}^{n} \phi_j.
\end{align*}
Let us study the Sobolev seminorms first, for which the triangle inequality yields that 
\begin{equation}\label{thm21pfstep2triineqsuper}
\begin{aligned}
    \left[\overline{\mathbf{s}}^{1/2}\nabla (v_m-\overline{u})\right]_{\Hminusul(\diamondsuit_n)}&\leq \left[\overline{\mathbf{s}}^{1/2}\nabla(u_{n+1}-\overline{u})\right]_{\Hminusul(\diamondsuit_n)}+\sum_{j=n+1}^{m-1}\left[\overline{\mathbf{s}}^{1/2}\nabla (u_{j+1}-u_j-\phi_j)\right]_{\Hminusul(\diamondsuit_n)} \\
    &+\sum_{j=1}^{n} \left[\overline{\mathbf{s}}^{1/2} \nabla\phi_j\right]_{\Hminusul(\diamondsuit_n)}.
\end{aligned}
\end{equation}
For the first term on the right-hand side above, the harmonic approximation presented in (\ref{thm21phstep2ptgoal}) gives the desired result immediately (that being the right-hand side of (\ref{thm21phstep2ptgoal})).
Since $w_j-\phi_j\in\mathcal{A}(\diamondsuit_j)\subset \mathcal{A}(\diamondsuit_{n})$ and $j\geq n+1$, we are then able to apply Lemma \ref{lemadaptpoinc}. This yields essentially that
\begin{align}\label{thm21pfstep2wjsmall}
    3^{-ns}\left[\overline{\mathbf{s}}^{1/2}\nabla (w_j-\phi_j) \right]_{\Hminusul(\diamondsuit_n)}\leq C
    \|\mathbf{s}^{1/2}\nabla (w_j-\phi_j)\|_{\underline{L}^2(\diamondsuit_n)}.
\end{align}
Then, we may utilize our assumption of (iii)$_{k-1}^{'}$ and \eqref{thm21pfstep2crudebd} at scale $j$ so that
\begin{align*}
    \|\mathbf{s}^{1/2}\nabla (w_j-\phi_j)\|_{\underline{L}^2(\diamondsuit_n)}\leq  C3^{-((j-1)-n)(k-1-\rho)}\|\mathbf{s}^{1/2}\nabla w_j\|_{\underline{L}^2(\diamondsuit_{j-1})}\leq C\left(\frac{\mathcal{X_H}}{3^j} \right)^{\nf{\theta}{2}}
    \|\mathbf{\overline{s}}^{1/2}\nabla \overline{u}\|_{\underline{L}^2(\diamondsuit_{j})}.
\end{align*}
Now, we note that $(\mathcal{X_H}/3^j)^{\nf{\theta}{2}}=3^{-\frac{\theta}{2}(j-n)} (\mathcal{X_H}/3^n)^{\nf{\theta}{2}}$, and going back the scales allow us to deduce for some $C(d,k,s,\gamma)<\infty$ that
\begin{align*}
    3^{-ns}\left[\overline{\mathbf{s}}^{1/2}\nabla (w_j-\phi_j) \right]_{\Hminusul(\diamondsuit_n)}\leq C3^{-(\nf{\theta}{2}-\rho)(j-n)}\left(\frac{\mathcal{X_H}}{3^n} \right)^{\nf{\theta}{2}}
    \|\mathbf{\overline{s}}^{1/2}\nabla \overline{u}\|_{\underline{L}^2(\diamondsuit_n)}.
\end{align*}
Because $\rho<\nf{\theta}{2}$, the sum in \eqref{thm21pfstep2triineqsuper} is uniformly bounded and thus, it holds that
\begin{align*}
    \sum_{j=n+1}^{m-1}\left[\overline{\mathbf{s}}^{1/2}\nabla (u_{j+1}-u_j-\phi_j)\right]_{\Hminusul(\diamondsuit_n)}\leq C\left(\frac{\mathcal{X_H}}{3^n} \right)^{\nf{\theta}{2}}
    \|\mathbf{\overline{s}}^{1/2}\nabla \overline{u}\|_{\underline{L}^2(\diamondsuit_n)}.
\end{align*}

Consequently, it remains to study the terms that include only the functions $\phi_j\in\mathcal{A}_{k-1}$ in (\ref{thm21pfstep2triineqsuper}). This will be the content of the following substep.

\emph{Step 2.2. Smallness of $\nabla \phi_j$.} Finally, in order to make sense of (\ref{thm21pfstep2triineqsuper}), we still need to argue that the terms consisting only of functions $\phi_j$ are small enough and that we can move to smaller scales $j\leq n$ within the weak norms. In other words, we wish to show that the last term of \eqref{thm21pfstep2triineqsuper} satisfies the upper bound of \eqref{thm21phstep2ptgoal}. The main idea here is to approximate the functions~$\phi_j$ by harmonic polynomials $p_j\in\overline{\mathcal{A}}_{k-1}$ of order $k-1$. Indeed, such polynomials exist by our induction assumption of~(i)$_{k-1}$, after which we obtain with the triangle inequality that
\begin{align}\label{thm21pfstep22triangle}
    \left[\overline{\mathbf{s}}^{1/2}\nabla \phi_j\right]_{\Hminusul(\diamondsuit_n)}\leq \left[\overline{\mathbf{s}}^{1/2}\nabla (\phi_j-p_j)\right]_{\Hminusul(\diamondsuit_n)}+\left[\overline{\mathbf{s}}^{1/2}\nabla p_j\right]_{\Hminusul(\diamondsuit_n)}.
\end{align}

Here, the first term on the right-hand side is bounded by the estimate given in (i)$_{k-1}$, so a change of scales gives that 
\begin{align}\label{thm31pfstep22harmpnoms}
    3^{-ns}\left[\overline{\mathbf{s}}^{1/2}\nabla (\phi_j-p_j)\right]_{\Hminusul(\diamondsuit_n)}\leq C3^{(k-2)(n-j)}\left(\frac{\mathcal{X_H}}{3^n} \right)^{\nf{\theta}{2}}\|\overline{\mathbf{s}}^{1/2}\nabla p_j \|_{\underline{L}^2(\diamondsuit_j)}
\end{align}
as the order of $\nabla p_j$ is $k-2$. Note that we take here such $k\geq 2$ so that $k-2\geq 0$ for the orders of the polynomials to be defined properly. This can be done because the case of $k=1$ is trivially valid due to $\phi_j\in \mathcal{A}_0$ being constants. The takeaway here is that the right-hand side of~(\ref{thm31pfstep22harmpnoms}) is a much more regular and integrable object than the left-hand side of (\ref{thm21pfstep22triangle}). It remains to estimate the second term on the right-hand side of (\ref{thm21pfstep22triangle}). Now, if we utilize the natural $L^2$ upper bound for negative Sobolev seminorms, we get, by the regularity properties of harmonic polynomials, that
\begin{equation}\label{thm21pfstep2phismall}
\begin{aligned}
    3^{-ns}\left[\overline{\mathbf{s}}^{1/2}\nabla p_j\right]_{\Hminusul(\diamondsuit_n)}&\leq C(d)\|\overline{\mathbf{s}}^{1/2}\nabla p_j\|_{\underline{L}^2(\diamondsuit_n)}\leq C3^{(k-2)(n-j)}\|\overline{\mathbf{s}}^{1/2}\nabla p_j\|_{\underline{L}^2(\diamondsuit_j)}.
\end{aligned}
\end{equation}
This implies, furthermore, that the terms consisting solely of the correctors $\phi_j$ in (\ref{thm21pfstep2triineqsuper}) are indeed small and integrable enough for our needs, since we can work very well with the harmonic polynomials $p_j$.

With these new regularity properties in mind, let us revert back to estimating the left-hand side of (\ref{thm21pfstep22triangle}). Now, for each $j\leq n$, we can freely utilize the triangle inequality with $w_j\in\mathcal{A}(\diamondsuit_j)$ to compute that
\begin{equation*}
\begin{aligned}
    3^{-ns}\left[\overline{\mathbf{s}}^{1/2}\nabla \phi_j\right]_{\Hminusul(\diamondsuit_n)}&\leq C3^{-js}3^{(k-2)(n-j)}\left[\overline{\mathbf{s}}^{1/2}\nabla \phi_j\right]_{\Hminusul(\diamondsuit_j)} \\
    &\leq C3^{-js}3^{(k-2)(n-j)}\left(\left[\overline{\mathbf{s}}^{1/2}\nabla (w_j-\phi_j)\right]_{\Hminusul(\diamondsuit_j)}+\left[\overline{\mathbf{s}}^{1/2}\nabla w_j\right]_{\Hminusul(\diamondsuit_j)}\right).
\end{aligned}
\end{equation*}
The first term on the right satisfies a rescaled variant of the estimate in (\ref{thm21pfstep2wjsmall}), whereas we can now reason for the last term with Lemma \ref{lemadaptpoinc} that
\begin{align*}
    3^{-js}\left[\overline{\mathbf{s}}^{1/2}\nabla w_j \right]_{\Hminusul(\diamondsuit_j)}\leq C\left(1+ C\left(\frac{\mathcal{X_H}}{3^j} \right)^{\nf{\theta}{2}}\right)\|\mathbf{s}^{1/2}\nabla w_j\|_{\underline{L}^2(\diamondsuit_j)}.
\end{align*}
From all the computations we have performed so far within Step 2.1, starting from (\ref{thm21pfstep2estim1}), it can now be seen that
\begin{align*}
    3^{-ns}\left[\overline{\mathbf{s}}^{1/2}\nabla \phi_j\right]_{\Hminusul(\diamondsuit_n)}\leq C3^{-(1-\nf{\theta}{2})(n-j)}\left(\frac{\mathcal{X_H}}{3^n} \right)^{\nf{\theta}{2}}
    \|\mathbf{\overline{s}}^{1/2}\nabla \overline{u}\|_{\underline{L}^2(\diamondsuit_{n})}.
\end{align*}

Thus, the last sum in (\ref{thm21pfstep2triineqsuper}) converges nicely, and we have concluded that each term there on the right-hand side is bounded by the same upper bound. In other words, this means that
\begin{align*}
    3^{-ns}\left[\overline{\mathbf{s}}^{1/2}\nabla (v_m-\overline{u})\right]_{\Hminusul(\diamondsuit_n)}\leq C\left(\frac{\mathcal{X_H}}{3^n} \right)^{\nf{\theta}{2}}
    \|\overline{\mathbf{s}}^{1/2}\nabla \overline{u} \|_{\underline{L}^2(\diamondsuit_{n})}.
\end{align*}
Consequently, we have now verified that the functions $u_n$ given by the harmonic approximation in~(\ref{thm21pfstep2harmappr}) are close enough to $v_m\in\mathcal{A}(\diamondsuit_m)$ above.

\emph{Step 2.3. Same for fluxes.} In this step, we repeat the same argument for the fluxes, that is, the left-hand side below (when scaled properly by the factor of $3^{-ns}$) has the same upper bound of~\eqref{thm21phstep2ptgoal}. Namely, it holds by (\ref{thm21pfstep2triineqsuper}) that 
\begin{equation}\label{thm21pfstep2fluxineqs}
    \begin{aligned}
&\left[\overline{\mathbf{s}}^{-1/2}(\mathbf{a}\nabla v_m-\overline{\mathbf{a}}\nabla \overline{u})\right]_{\Hminusul(\diamondsuit_n)}
\leq \left[\overline{\mathbf{s}}^{-1/2}(\mathbf{a}\nabla u_{n+1}-\overline{\mathbf{a}}\nabla \overline{u})\right]_{\Hminusul(\diamondsuit_n)} \\ 
&+\sum_{j=n+1}^{m-1}\left[\overline{\mathbf{s}}^{-1/2}\mathbf{a}\nabla(u_{j+1}-u_j-\phi_j)\right]_{\Hminusul(\diamondsuit_n)}+\sum_{j=1}^{n}\left[\overline{\mathbf{s}}^{-1/2}\mathbf{a}\nabla\phi_j\right]_{\Hminusul(\diamondsuit_n)}.
    \end{aligned}
\end{equation}
For the first term on the right-hand side, the harmonic approximation in (\ref{thm21phstep2ptgoal}) gives directly the desired estimate again, that is,
\begin{align}\label{thm21pfstep23goodrhs}
    3^{-ns}\left[\overline{\mathbf{s}}^{-1/2}(\mathbf{a}\nabla u_{n+1}-\overline{\mathbf{a}}\nabla \overline{u})\right]_{\Hminusul(\diamondsuit_n)}\leq C\left(\frac{\mathcal{X_H}}{3^n} \right)^{\nf{\theta}{2}}
    \|\overline{\mathbf{s}}^{1/2}\nabla \overline{u} \|_{\underline{L}^2(\diamondsuit_n)}.
\end{align}
For the second term, we utilize Lemma \ref{lemadaptpoinc} to obtain that
\begin{align}\label{thm21pfstep23middleterm}
3^{-ns} \left[\overline{\mathbf{s}}^{-1/2}\mathbf{a}\nabla(w_j-\phi_j)\right]_{\Hminusul(\diamondsuit_n)}\leq C\left(\frac{\mathcal{X_H}}{3^n} \right)^{\nf{\theta}{2}}\|\mathbf{s}^{1/2}\nabla(w_j-\phi_j)\|_{\underline{L}^2(\diamondsuit_n)}+C\|\mathbf{s}^{1/2}\nabla(w_j-\phi_j)\|_{\underline{L}^2(\diamondsuit_n)}
\end{align}
after which, we can follow our earlier reasoning in Step 2.1, starting from (\ref{thm21pfstep2estim1}) to obtain the same right-hand side as in (\ref{thm21pfstep23goodrhs}). For the final term in the second row of (\ref{thm21pfstep2fluxineqs}), where again $j\leq n$, we can write similarly as before with the triangle inequality that
\begin{equation}\label{thm21pfstep2fluxphi}
\begin{aligned}
3^{-ns}&\left[\overline{\mathbf{s}}^{-1/2}\mathbf{a}\nabla\phi_j\right]_{\Hminusul(\diamondsuit_n)}\leq C3^{-js}3^{(k-2)(n-j)}\left[\overline{\mathbf{s}}^{-1/2}\mathbf{a}\nabla\phi_j\right]_{\Hminusul(\diamondsuit_j)} \\
&\leq C3^{-js}3^{(k-2)(n-j)}\left(\left[\overline{\mathbf{s}}^{-1/2}\mathbf{a}\nabla(w_j-\phi_j)\right]_{\Hminusul(\diamondsuit_j)}+\left[\overline{\mathbf{s}}^{-1/2}\mathbf{a}\nabla w_j\right]_{\Hminusul(\diamondsuit_j)}\right),
\end{aligned}
\end{equation}
whose first term on the right (up to a change of scales) we already estimated in (\ref{thm21pfstep23middleterm}). For the second term, we can first utilize Lemma \ref{lemadaptpoinc} and then our earlier computations starting from (\ref{thm21pfstep21nic}). Again, we can argue that the left-hand side of (\ref{thm21pfstep2fluxphi}) is sufficiently small by~(i)$_{k-1}$, the triangle inequality, Lemma \ref{lemadaptpoinc}, and (\ref{fromutouhom}). In summary, we can now state that
\begin{align*}
    3^{-ns}\left[\overline{\mathbf{s}}^{-1/2}(\mathbf{a}\nabla v_m-\overline{\mathbf{a}}\nabla \overline{u})\right]_{\Hminusul(\diamondsuit_n)}\leq C\left(\frac{\mathcal{X_H}}{3^n} \right)^{\nf{\theta}{2}}
    \|\overline{\mathbf{s}}^{1/2}\nabla \overline{u} \|_{\underline{L}^2(\diamondsuit_n)}.
\end{align*}

\emph{Step 2.4. One last time for the $L^2$ norms.} We will repeat the same argument once more for the~$L^2$ norms to show that the left-hand side below (after proper scaling) satisfies the upper bound of~\eqref{thm21phstep2ptgoal}. Imitating \eqref{thm21pfstep2triineqsuper}, the triangle inequality yields that
\begin{align}\label{step24triineq}
    \|v_m-\overline{u}\|_{\underline{L}^2(\diamondsuit_n)}\leq \|u_{n+1}-\overline{u}\|_{\underline{L}^2(\diamondsuit_n)}+\sum_{j=n+1}^{m-1}\|w_j-\phi_j\|_{\underline{L}^2(\diamondsuit_n)}+\sum_{j=1}^{n}\|\phi_j\|_{\underline{L}^2(\diamondsuit_n)}.
\end{align}
Again, for the first term, the estimation of \eqref{thm21phstep2ptgoal} directly provides the result that
\begin{align*}
    3^{-n}\overline{\lambda}^{1/2}\|u_{n+1}-\overline{u}\|_{\underline{L}^2(\diamondsuit_{n})}\leq C\left(\frac{\mathcal{X_H}}{3^n} \right)^{\nf{\theta}{2}}
    \|\overline{\mathbf{s}}^{1/2}\nabla \overline{u} \|_{\underline{L}^2(\diamondsuit_n)}.
\end{align*}
For the second term, we apply the Poincaré inequality (after normalizing the functions properly so that $(w_j-\phi_j)_{\diamondsuit_n}=0$), stating that
\begin{align*}
    3^{-n}\overline{\lambda}^{1/2}\|w_j-\phi_j\|_{\underline{L}^2(\diamondsuit_n)}\leq C\|\shom^{1/2}\nabla(w_j-\phi_j)\|_{\underline{L}^2(\diamondsuit_n)}
\end{align*}
after which, the claim follows by Step 2.1. For the last term on the right-hand side of \eqref{step24triineq}, we apply the triangle inequality to achieve that
\begin{align*}
    3^{-n}\overline{\lambda}^{1/2}\|\phi_j\|_{\underline{L}^2(\diamondsuit_n)}\leq C3^{-j}\overline{\lambda}^{1/2}3^{(k-2)(n-j)}\left(\|w_j-\phi_j\|_{\underline{L}^2(\diamondsuit_j)}+\|w_j\|_{\underline{L}^2(\diamondsuit_j)}\right)
\end{align*}
as now $j\leq n$. For the first term on the right-hand side, the Poincaré inequality and Step 2.1 again yield the desired upper bound. For the second term, it suffices to apply the triangle inequality and the harmonic approximation estimate. Furthermore, utilizing (i)$_{k-1}$, we can argue as before that the norms within the last summation of \eqref{step24triineq} are small and regular enough. Consequently, we can now conclude that
\begin{align*}
    3^{-n}\overline{\lambda}^{1/2}\|v_m-\overline{u}\|_{\underline{L}^2(\diamondsuit_n)}\leq C\left(\frac{\mathcal{X_H}}{3^n} \right)^{\nf{\theta}{2}}
    \|\overline{\mathbf{s}}^{1/2}\nabla \overline{u} \|_{\underline{L}^2(\diamondsuit_n)}.
\end{align*}

\emph{Step 2.5. Conclusion.} After all of the previous substeps, we finally have the desired estimate for the expression on the left-hand side below. Namely, we have shown that the following bound can be estimated for some constant $C(d,k,s,\gamma)<\infty$ after taking the limit $v_m\to u$ as $m\to\infty$ that
\begin{align*}
    3^{-n}\overline{\lambda}^{1/2}\|u-\overline{u} \|_{\underline{L}^2(\diamondsuit_n)}+3^{-ns}\left[\overline{\mathbf{A}}^{1/2}\begin{bmatrix}
\nabla u-\nabla \overline{u} \\ \mathbf{a}\nabla u-\overline{\mathbf{a}}\nabla\overline{u}
    \end{bmatrix} \right]_{\Hminusul(\diamondsuit_n)}\leq C
    \left(\frac{\mathcal{X_H}}{3^n} \right)^{\nf{\theta}{2}} \|\overline{\mathbf{s}}^{1/2}\nabla \overline{u} \|_{\underline{L}^2(\diamondsuit_n)}.
\end{align*}
This limit, indeed, exists due to the known precompactness and embedding results, but let us be more precise here and justify this fact more rigorously.

Since the functions $v_m$ are locally bounded in $L^2$, we know due to precompactness that a weak vector-valued limit $\mathbf{f}$ exists (up to a subsequence) for $\mathbf{s}^{1/2}\nabla v_m$ as $m\to\infty$. On the other hand, there is also a strong limit $v_m\to v$ within $L^2$ functions. It remains to be seen that this limit $v$ is actually the same function $u$ as in the claim of (\ref{thmhighregclaim2}) and that it solves the original elliptic equation of (\ref{ellipeq}).

Let us start by showing that this limit $v$ coincides with our claim (ii)$_k$ to be proven. Fix any test function $\varphi\in C_c^\infty(\diamondsuit_n)$. Note that the scale of the adapted cube here is nearly irrelevant, as is whether we utilize volume-normalized integrals below or not. Integrating by parts provides that
\begin{align*}
    \int_{\diamondsuit_n} \varphi\cdot\nabla v_m=-\int_{\diamondsuit_n} \nabla\cdot\varphi v_m \xrightarrow[m\to\infty]{} -\int_{\diamondsuit_n} \nabla\cdot\varphi v.
\end{align*}
By rewriting the left-hand side above, we may compute that
\begin{align*}
    \int_{\diamondsuit_n} \mathbf{s}^{-1/2}\varphi\cdot\mathbf{s}^{1/2}\nabla v_m \xrightarrow[m\to\infty]{} \int_{\diamondsuit_n} \mathbf{s}^{-1/2}\varphi\cdot \mathbf{f} = \int_{\diamondsuit_n} \varphi\cdot\mathbf{s}^{-1/2}\mathbf{f}.
\end{align*}
Combining the two displays above allows us to deduce that $\nabla v=\mathbf{s}^{-1/2}\mathbf{f}$ in the sense of distributions as $\mathbf{s}^{-1/2}\in L^2$ locally. That is, as we already showed in the previous substeps, we have now found a limiting function $v$, which satisfies the claim of (ii)$_k$, so it turns out that indeed $v=u$ necessarily.

The final observation to note here is that the function $v$ discovered above is a solution of the elliptic equation in (\ref{ellipeq}). Again, we fix some test function $\varphi\in C_c^\infty(\diamondsuit_n)$. Because each function $v_m$ is a solution of (\ref{ellipeq}), we may reason that
\begin{align*}
    \int_{\diamondsuit_n} \mathbf{a}\nabla v_m\cdot \nabla\varphi = \int_{\diamondsuit_n} \mathbf{a}\mathbf{s}^{-1/2}\mathbf{s}^{1/2}\nabla v_m\cdot \nabla\varphi = \int_{\diamondsuit_n} \mathbf{s}^{1/2}\nabla v_m\cdot \mathbf{s}^{-1/2}\mathbf{a}^t\nabla\varphi = 0.
\end{align*}
Here $\mathbf{s}^{-1/2}\mathbf{a}^t\nabla\varphi\in L_{\textnormal{loc}}^2$, and because $\a^t=\s-\k$, we have that
\begin{align*}
    \mathbf{s}^{-1/2}\mathbf{a}^t\nabla\varphi = \s^{1/2}\nabla\varphi - \mathbf{s}^{-1/2}\k\nabla\varphi,
\end{align*}
whose both terms belong to $L_{\textnormal{loc}}^2$ by Assumption \ref{d.renormellipt}. Thus, we may take the limit $m\to\infty$, which allows us to conclude that
\begin{align*}
    \int_{\diamondsuit_n} \mathbf{s}^{1/2}\nabla v\cdot \mathbf{s}^{-1/2}\mathbf{a}^t\nabla\varphi = \int_{\diamondsuit_n} \mathbf{a}\nabla v\cdot \nabla\varphi = 0
\end{align*}
implying that $v$ is a solution of (\ref{ellipeq}). This completes the proof of (ii)$_k$.

\emph{Step 3. Proof of (i)$_k$.} Suppose now that the assertions (i)$_{k-1}$ and (ii)$_k$ hold. From the harmonic approximation result of Proposition \ref{propharmappr}, it can be seen that for each $u\in\mathcal{A}_k\subset\mathcal{A}(\R^d)$, which we restrict to $\diamondsuit_n$, there exists some $\overline{u}\in\overline{\mathcal{A}}(\diamondsuit_{n})$ such that
\begin{align*}
    3^{-n}\overline{\lambda}^{1/2}\|u-\overline{u}\|_{\underline{L}^2(\diamondsuit_{n})}+3^{-ns}\left[\overline{\mathbf{A}}^{1/2}\begin{bmatrix}
       \nabla u- \nabla \overline{u} \\ \mathbf{a}\nabla u -\overline{\mathbf{a}}\nabla\overline{u}
    \end{bmatrix} \right]_{\Hminusul(\diamondsuit_{n})}\leq C\left(\frac{\mathcal{X_H}}{3^n} \right)^{\nf{\theta}{2}} \|\mathbf{s}^{1/2}\nabla u \|_{\underline{L}^2(\diamondsuit_n)}.
\end{align*}
Thus, it will be enough to ''upgrade'' $\overline{u}$ from $\overline{\mathcal{A}}(\diamondsuit_{n})$ to $\overline{\mathcal{A}}_k$ with a similar argument to that in the previous step. This implies the claim by applying our reasoning from \eqref{fromutouhom} for the right-hand side above. That is, we want to show that for every $u\in\mathcal{A}_k$, there exists a certain $\overline{u}\in\overline{\mathcal{A}}_k$ such that the following estimate
\begin{align}\label{thm21pfstep3goal}
    3^{-n}\overline{\lambda}^{1/2}\|u-\overline{u}\|_{\underline{L}^2(\diamondsuit_{n})}+3^{-ns}\left[\overline{\mathbf{A}}^{1/2}\begin{bmatrix}
       \nabla u- \nabla \overline{u} \\ \mathbf{a}\nabla u -\overline{\mathbf{a}}\nabla\overline{u}
    \end{bmatrix} \right]_{\Hminusul(\diamondsuit_{n})}\leq C\left(\frac{\mathcal{X_H}}{3^n} \right)^{\nf{\theta}{2}} \|\overline{\mathbf{s}}^{1/2}\nabla \overline{u} \|_{\underline{L}^2(\diamondsuit_n)}
\end{align}
holds for some constant $C(d,k,s,\gamma)<\infty$ as the scale is sufficiently large (that is, $3^n\geq \mathcal{X_H})$. 

\emph{Step 3.1. Analyzing the excess decay.} As in \cite{ak}, to prove the desired result of (\ref{thm21pfstep3goal}), we wish to study the decay of the following \emph{excess} quantity defined by
\begin{align*}
    E_k(r):=r^{-k}\inf_{p\in\mathcal{\overline{A}}_k}\|u-p\|_{\underline{L}^2(B_{r})},
\end{align*}
where $r\geq C(d,k,\mu)\mathcal{X_H}$, for a large enough $C(d,k,\mu)<\infty$ specified below, and $u\in\mathcal{A}_k$ are fixed. Note that we can traverse between the adapted balls and cubes interchangeably. We recall that here and elsewhere (unless stated otherwise)~$B_r$ refers to the adapted open ball defined in (\ref{defadaptedballs}). Now, our aim is to prove the following decay estimate for some fixed constant $\mu(d)\in (0,\nf{1}{2}]$ so that
\begin{align}\label{thm21pfexcessestimate}
    E_k(\mu r)\leq \frac{1}{2}E_k(r)+
    C(d,k,s,\gamma,\mu)r^{-k}\left(\frac{r}{\mathcal{X_H}} \right)^{-\nf{\theta}{2}}\|p_{k,r} \|_{\underline{L}^2(B_{r})},
\end{align}
where $p_{k,r}\in\overline{\mathcal{A}}_k$ is the $\mathbf{\overline{a}}$-harmonic polynomial minimizer for the quantity $E_k(r)$. Indeed, by employing the triangle inequality, we have that
\begin{align}\label{step31triineq}
    E_k(\mu r)\leq (\mu r)^{-k}\inf_{p\in\mathcal{\overline{A}}_k}\left(\|u-\overline{u}\|_{\underline{L}^2(B_{\mu r})}+\|\overline{u}-p\|_{\underline{L}^2(B_{\mu r})} \right).
\end{align}

For the second term above, the properties of harmonic polynomials allow us to estimate with the triangle inequality that
\begin{align}\label{thm21pfstep3decayfirstterm}
    (\mu r)^{-k}\inf_{p\in\mathcal{\overline{A}}_k}\|\overline{u}-p\|_{\underline{L}^2(B_{\mu r})}\leq C\mu r^{-k} \inf_{p\in\mathcal{\overline{A}}_k}\|\overline{u}-p\|_{\underline{L}^2(B_{r})}\leq C\mu E_k(r)+C\mu r^{-k}\|u-\overline{u}\|_{\underline{L}^2(B_{r})}
\end{align}
for some positive $C(d,k)<\infty$. Here, we will choose $\mu$ so that $C\mu<\nf16$ for the largest constant~$C$ above. This essentially yields the first term of (\ref{thm21pfexcessestimate}). For the first term above on the right-hand side of \eqref{step31triineq}, it follows from the established harmonic approximation estimate of Proposition~\ref{propharmappr} that
\begin{align*}
    (\mu r)^{-k}\|u-\overline{u}\|_{\underline{L}^2(B_{\mu r})}
    \leq C\mu^{-k+1-\nf{\theta}{2}} r^{-k+1}\overline{\lambda}^{-1/2}\left(\frac{\mathcal{X_H}}{r} \right)^{\nf{\theta}{2}} \|\mathbf{s}^{1/2}\nabla u\|_{\underline{L}^2(B_{\nf{r}{2}})}.
\end{align*}
Next, by applying the non-iterative Caccioppoli inequality for adapted balls in Corollary \ref{cor21}, we have that
\begin{align*}
    (\mu r)^{-k}\|u-\overline{u}\|_{\underline{L}^2(B_{\mu r})}\leq C\mu^{-k+1-\nf{\theta}{2}} r^{-k}\left(\frac{\mathcal{X_H}}{r} \right)^{\nf{\theta}{2}} \|u-(u)_{B_r} \|_{\underline{L}^2(B_{r})}.
\end{align*}
We now note that due to the triangle inequality and $(u)_{B_r}=(u-p)_{B_r}+(p)_{B_r}$, we may reason that
\begin{align*}
    \|u-(u)_{B_r} \|_{\underline{L}^2(B_{r})}\leq \|u-p \|_{\underline{L}^2(B_{r})}+\|p-(p)_{B_r} \|_{\underline{L}^2(B_{r})},
\end{align*}
and so, by taking the infimum over $p\in\overline{\mathcal{A}}_k$, we can conclude that
\begin{align*}
    (\mu r)^{-k}\|u-\overline{u}\|_{\underline{L}^2(B_{\mu r})}\leq C\mu^{-k+1-\nf{\theta}{2}}\left(\frac{r}{\mathcal{X_H}} \right)^{\nf{-\theta}{2}} E_k(r)+C\mu^{-k+1-\nf{\theta}{2}} r^{-k}\left(\frac{r}{\mathcal{X_H}} \right)^{\nf{-\theta}{2}}\|p_{k,r} \|_{\underline{L}^2(B_{r})}.
\end{align*}

We can then choose the scale $r\geq C(d,k,\mu)\mathcal{X_H}$ so big that 
\begin{align*}
    (\mu r)^{-k}\|u-\overline{u}\|_{\underline{L}^2(B_{\mu r})}\leq \frac{1}{6}E_k(r)+C\mu^{-k+1-\nf{\theta}{2}} r^{-k}\left(\frac{r}{\mathcal{X_H}} \right)^{\nf{-\theta}{2}}\|p_{k,r} \|_{\underline{L}^2(B_{r})}
\end{align*}
always holds for $r\in[C(d,k,\mu)\mathcal{X_H},\infty)$. Consequently, the desired decay estimate of (\ref{thm21pfexcessestimate}) follows by reabsorbing the first term on the right and performing the same analysis for the last term of~\eqref{thm21pfstep3decayfirstterm} as above. We also note that the same argument applies for every $k\in\N$.

\emph{Step 3.2. Integrating and estimating the excess.} Let us then revert back to the statement in~(\ref{thm21pfstep3goal}). The aim is to integrate the estimate of (\ref{thm21pfexcessestimate}) with respect to the Haar measure on~$\R_+$ and to proceed inductively for $k$. We wish to establish an upper bound in the form of \eqref{thm21pfstep32endpt}. Let $r,R>0$ be such that $r\in[C(d,k,\mu)\mathcal{X_H},\infty)$ and $\frac{r}{\mu}\leq R$. A change of variables within the Haar measure provides that
\begin{equation*}
\begin{aligned}
    \int_{\mu r}^{\mu R} E_{k+1}(t) \frac{dt}{t}&=\int_r^R E_{k+1}(\mu t) \frac{dt}{t} \leq \frac{1}{2}\int_r^{\mu R} E_{k+1}(t)\frac{dt}{t} \\ &+\frac{1}{2}\int_{\mu R}^R E_{k+1}(t)\frac{dt}{t}+C\int_r^R \left(\frac{t}{\mathcal{X_H}} \right)^{\nf{-\theta}{2}} \|p_{k+1,t}\|_{\underline{L}^2(B_{t})} \frac{dt}{t^{k+2}},
\end{aligned}
\end{equation*}
where the second integral on the rightmost side vanishes as $R\to\infty$ since $u\in\mathcal{A}_k$ and $p\in\overline{\mathcal{A}}_{k+1}$. Furthermore, by reabsorbing the first term from the right-hand side to the left-hand side, we can now deduce that
\begin{align}\label{thm21pfstep3estim1}
    \int_{\mu r}^{\infty} E_{k+1}(t) \frac{dt}{t}\leq C\int_r^\infty \left(\frac{t}{\mathcal{X_H}} \right)^{\nf{-\theta}{2}} \|p_{k+1,t}\|_{\underline{L}^2(B_{t})} \frac{dt}{t^{k+2}}.
\end{align}

Consider next the Taylor series of $p\in\mathcal{\overline{A}}_k$. For all $l\in\N$, we define the mapping
\begin{align*}
    \pi_lp(x):=\sum_{|\alpha|=l}\frac{1}{\alpha !}(\partial_\alpha p)(0)x^\alpha
\end{align*}
to portray the different homogeneity parts of $p$ exactly as in Appendix A. Furthermore, let us define that 
\begin{equation}\label{thm21pfstep3defomega}
\begin{aligned}
    \omega(r):=\sum_{l=0}^k \omega_l(r) \quad \textnormal{with} \quad \omega_l(r):\!&=\int_r^\infty \left(\frac{r}{t} \right)^{k+1+\nf{\theta}{2}}\|\pi_lp_{k,t} \|_{\underline{L}^2(B_{t})} \frac{dt}{t}.
    \end{aligned}
\end{equation}
Note that $p_{k,t}$ is an $\overline{\mathbf{a}}$-harmonic polynomial of order $k$, so it possesses the orthogonality properties of harmonic polynomials within the adapted balls, as explained in Appendix A. Thus, we can utilize the \emph{spherical harmonics} theory and perform a change of variables (cf. Appendix A) in order to employ the results for the usual harmonic polynomials of the standard Euclidean geometry. Consequently, this yields the upper bound due to (\ref{appendixsphericalharmonics}) that
\begin{align*}
    \omega(r)\leq \int_r^\infty \left(\frac{r}{t} \right)^{k+1+\nf{\theta}{2}}\|p_{k,t} \|_{\underline{L}^2(B_{t})} \frac{dt}{t}.
\end{align*}

By adopting this notation and applying the triangle inequality after dividing the polynomial~$p_{k+1,t}$ into its first $k$ and $k+1$ parts of their respective orders (as we identify the lower-order parts of $p_{k+1,t}$ with $p_{k,t}$), the estimate in (\ref{thm21pfstep3estim1}) provides that
\begin{align}\label{thm21pfstep3intenest}
    \int_{\mu r}^{\infty} E_{k+1}(t) \frac{dt}{t}\leq C\left(\frac{r}{\mathcal{X_H}}\right)^{\nf{-\theta}{2}} r^{-k-1}\omega(r)+C\int_r^\infty \left(\frac{t}{\mathcal{X_H}} \right)^{\nf{-\theta}{2}} \|\pi_{k+1} p_{k+1,t}\|_{\underline{L}^2(B_t)} \frac{dt}{t^{k+2}}.
\end{align}
Now, we wish to argue that the last term on the right-hand side above becomes sufficiently small. Indeed, if $s,t\in\R$, then we have by a telescope summation argument that
\begin{equation}\label{thm21pfstep3insert1}
\begin{aligned}
    t^{-k-1}\|\pi_{k+1} p_{k+1,t}\|_{\underline{L}^2(B_t)}&\leq \sum_{j=1}^\infty t^{-k-1}\|\pi_{k+1} p_{k+1,2^jt}-\pi_{k+1} p_{k+1,2^{j-1}t}\|_{\underline{L}^2(B_{t})}\\
    &+\limsup_{s\to\infty} t^{-k-1}\|\pi_{k+1} p_{k+1,s}\|_{\underline{L}^2(B_{t})}.
\end{aligned}    
\end{equation}
For the second term on the right, we apply the spherical harmonics theory after a change of scales for $t\leq s$, which yields that
\begin{align*}
    t^{-k-1}\|\pi_{k+1}p_{k+1,s}\|_{\underline{L}^2(B_{t})}\leq s^{-k-1}\|\pi_{k+1}p_{k+1,s}\|_{\underline{L}^2(B_{s})}\leq s^{-k-1}\|p_{k+1,s}\|_{\underline{L}^2(B_{s})}.
\end{align*}
Furthermore, the triangle inequality and $p_{k+1,s}$ being the harmonic polynomial minimizer for $u$ allow us to estimate that
\begin{align*}
    s^{-k-1}\|p_{k+1,s}\|_{\underline{L}^2(B_{s})}\leq s^{-k-1}\|u-p_{k+1,s}\|_{\underline{L}^2(B_{s})}+s^{-k-1}\|u\|_{\underline{L}^2(B_{s})}\leq 2s^{-k-1}\|u\|_{\underline{L}^2(B_{s})}.
\end{align*}
Because $u\in\mathcal{A}_k$, this implies that $s^{-k-1}\|u\|_{\underline{L}^2(B_{s})}$ tends to zero as $s\to\infty$, so it holds at the limit that the second term on the right of \eqref{thm21pfstep3insert1} also vanishes. 

For the first term consisting of the telescope summation in (\ref{thm21pfstep3insert1}), the same reasoning remains valid as above. That is, by applying the triangle inequality multiple times, we can conclude its decay as well. Consequently, we are now able to state that the integrand of the following integral
\begin{align*}
    \int_r^\infty \left(\frac{t}{\mathcal{X_H}} \right)^{\nf{-\theta}{2}} \|\pi_{k+1}p_{k+1,t}\|_{\underline{L}^2(B_t)} \frac{dt}{t^{k+2}}
\end{align*}
decays to zero as $t\to\infty$. To summarize our conclusions so far, we first note that the triangle inequality, as well as the definitions of $E_{k+1}$ and $p_{k+1,t}$ for $0<t\leq s\leq 2t$, imply that
\begin{equation}\label{thm21pfstep3pdiff}
\begin{aligned}
    s^{-k-1}\|p_{k+1,s}-p_{k+1,t}\|_{\underline{L}^2(B_t)}&\leq s^{-k-1}\left(\|u-p_{k+1,s}\|_{\underline{L}^2(B_t)}+\|u-p_{k+1,t}\|_{\underline{L}^2(B_t)}\right)\leq CE_{k+1}(2t).
\end{aligned}
\end{equation}
Furthermore, by applying a telescope summation argument once more alongside the triangle inequality and our earlier reasoning from \eqref{thm21pfstep3insert1} to (\ref{thm21pfstep3pdiff}), we can now deduce for the Haar measures that
\begin{equation}\label{thm21pfstep3insert3}
\begin{aligned}
    \sup_{t\in[r,\infty)}t^{-k-1}\|\pi_{k+1}p_{k+1,t}\|_{\underline{L}^2(B_t)}&\leq C\int_{r/2}^\infty \sup_{s\in(t,2t)}t^{-k-1}\|\pi_{k+1}p_{k+1,s}-\pi_{k+1}p_{k+1,t}\|_{\underline{L}^2(B_t)}\frac{dt}{t} \\
    &\leq C\int_{\mu r}^\infty E_{k+1}(t) \frac{dt}{t}.
\end{aligned}
\end{equation}
Now, by inserting the estimate from (\ref{thm21pfstep3insert3}) into the established estimate of (\ref{thm21pfstep3intenest}), we obtain that
\begin{align}\label{thm21pfstep3summary}
    \int_{\mu r}^{\infty} E_{k+1}(t) \frac{dt}{t}\leq C\int_{\mu r}^\infty \left(\frac{t}{\mathcal{X_H}} \right)^{\nf{-\theta}{2}} \frac{dt}{t}\int_{\mu r}^\infty E_{k+1}(t) \frac{dt}{t}+C\left(\frac{r}{\mathcal{X_H}} \right)^{\nf{-\theta}{2}}r^{-k-1}\omega(r)
\end{align}
holds while $C(d,k,\mu)\mathcal{X_H}\leq r$. Consequently, the inequality above satisfies for large enough $r$ after reabsorbing by (\ref{thm21pfstep3insert3}) that
\begin{align}\label{thm21pfstep32endpt}
    \int_{\mu r}^{\infty} E_{k+1}(t) \frac{dt}{t}+\sup_{t\in[r,\infty)}t^{-k-1}\|\pi_{k+1}p_{k+1,t}\|_{\underline{L}^2(B_t)}\leq C\left(\frac{r}{\mathcal{X_H}} \right)^{\nf{-\theta}{2}}r^{-k-1}\omega(r).
\end{align}

\emph{Step 3.3. Estimating $\omega(r)$.} Our next objective here is to assess the size of $\omega(r)$ above. We will proceed by analyzing each homogeneous part $\omega_l(r)$ separately to conclude \eqref{thm21pfstep3midpt} at the end of this step. The definition of $\omega_l(r)$ with a change of variables yields that
\begin{equation*}
\begin{aligned}
    \omega_l(2r)&=\int_{2r}^\infty \left(\frac{2r}{t} \right)^{k+1+\nf{\theta}{2}}\|\pi_l p_{k,t} \|_{\underline{L}^2(B_{t})} \frac{dt}{t} 
    =2^l\int_r^\infty \left(\frac{r}{t} \right)^{k+1+\nf{\theta}{2}} \|\pi_l p_{k,2t} \|_{\underline{L}^2(B_{t})} \frac{dt}{t}.
\end{aligned}
\end{equation*}
Furthermore, with the triangle inequality argument starting from (\ref{thm21pfstep3pdiff}), we deduce with \eqref{thm21pfstep32endpt} that
\begin{equation*}
\begin{aligned}
    \omega_l(2r)&\leq 2^l\omega_l(r)+C\int_r^\infty \left(\frac{r}{t} \right)^{k+1+\nf{\theta}{2}} \|p_{k,2t}-p_{k,t} \|_{\underline{L}^2(B_{t})} \frac{dt}{t} \\
    &\leq 2^l\omega_l(r)+Cr^{k+1}\int_r^\infty \left(\frac{r}{t} \right)^{\nf{\theta}{2}}\left(E_{k+1}(t)+\sup_{t\in[r,\infty)}t^{-k-1}\|\pi_{k+1}p_{k+1,t}\|_{\underline{L}^2(B_t)} \right) \frac{dt}{t}.
\end{aligned}
\end{equation*}
Consequently, we can now conclude due to our earlier computations in (\ref{thm21pfstep32endpt}) that
\begin{align*}
    \omega_l(2r)\leq 2^l\omega_l(r)+C\left(\frac{r}{\mathcal{X_H}} \right)^{\nf{-\theta}{2}}\omega(r).
\end{align*}

Applying the following summation and iteration argument with respect to $r$, we can estimate for every $r,t\in[C(d,k,\mu)\mathcal{X_H},\infty)$ with $r\leq t$ that
\begin{align}\label{thm21pfstep3omegaests}
    \omega(t)\leq C\left(\frac{t}{r} \right)^k\omega(r)\quad \textnormal{and} \quad \sum_{l=0}^{k-1} \omega_l(t)\leq C\left(\frac{t}{r} \right)^{k-\nf{\theta}{2}}\omega(r).
\end{align}
Let us briefly justify the statements above. Namely, if we write that $f(r):=r^{-k}\omega(r)$, then this implies the inequality that
\begin{align*}
    f(2r)\leq \left(1+C\left(\frac{r}{\mathcal{X_H}} \right)^{\nf{-\theta}{2}} \right)f(r).
\end{align*}
After iterating and taking the logarithm, we can see that
\begin{align*}
    \log\left(\frac{f(2^mr)}{f(r)} \right)\leq \sum_{j=0}^{m-1} \log\left(1+C\left(\frac{2^jr}{\mathcal{X_H}} \right)^{\nf{-\theta}{2}} \right),
\end{align*}
where the sum on the right-hand side is clearly bounded. This means that $f(2^mr)\leq Cf(r)$ holds for each $m\in\N$. The first claim then follows by setting $2^{m-1}r<t\leq 2^mr$. For the second claim, we utilize the definition of $\omega(r)$ and iterate across the scales.

Now, by the definition of $E_k$ and the Haar measure along with a change of variables, we note that
\begin{align*}
    E_k(r)\leq Cr\left(\int_{r}^{2r} E_{k+1}(t) \frac{dt}{t}+\sup_{t\in[r,2r]}t^{-k-1}\|\pi_{k+1} p_{k+1,t}\|_{\underline{L}^2(B_t)} \right),
\end{align*}
which allows us to deduce from (\ref{thm21pfstep3summary}) onward and (\ref{thm21pfstep3omegaests}) that
\begin{align}\label{thm21pfstep3midpt}
    \int_r^\infty E_k(t)\frac{dt}{t}\leq C\int_r^\infty \left(\frac{t}{\mathcal{X_H}} \right)^{\nf{-\theta}{2}} t^{-k}\omega(t) \frac{dt}{t}\leq C\left(\frac{r}{\mathcal{X_H}} \right)^{\nf{-\theta}{2}} r^{-k}\omega(r).
\end{align}
Looking back to (\ref{thm21pfstep3intenest}), we can now reduce the degrees of harmonic polynomials from $k+1$ to~$k$, since once again $u\in\mathcal{A}_k$.

\emph{Step 3.4. Conclusion.} Our deductions so far allow us to consider the homogeneous polynomial $\Tilde{p}_k\in\mathcal{\overline{A}}_k$ as a limit $\pi_k \Tilde{p}_k:=\lim_{t\to\infty} \pi_k p_{k,t}$ of homogeneous polynomials. In order to prove that this limit converges, we need to show that $\{\pi_k p_{k,t} \}_t$ is a Cauchy sequence. For each $r\in[C(d,k,\mu)\mathcal{X_H},\infty)$, we obtain by the triangle inequality, telescope summation, the Haar measure, \eqref{thm21pfstep3insert3}, and (\ref{thm21pfstep3midpt}) that
\begin{equation*}
\begin{aligned}
    \sup_{t\in(2r,\infty)}r^{-k}\|\pi_k p_{k,t}-\pi_k p_{k,r}\|_{\underline{L}^2(B_r)}&\leq C\int_r^\infty \sup_{s\in(t,2t)} t^{-k}\|\pi_kp_{k,s}-\pi_kp_{k,t}\|_{\underline{L}^2(B_t)} \frac{dt}{t} \\
    &\leq C\int_r^\infty E_k(t) \frac{dt}{t}\leq C\left(\frac{r}{\mathcal{X_H}} \right)^{\nf{-\theta}{2}}r^{-k}\omega(r).
\end{aligned}
\end{equation*}
It is now clear that the rightmost side vanishes as $r\to\infty$, and so the desired homogeneous polynomial $\Tilde{p}_k\in\mathcal{\overline{A}}_k$ exists as the aforementioned Cauchy sequence converges. In conclusion, now we have for each $r\in[C(d,k,\mu)\mathcal{X_H},\infty)$ that
\begin{align*}
    r^{-k}\|\pi_k p_{k,r}-\pi_k \Tilde{p}_k\|_{\underline{L}^2(B_r)}\leq C\left(\frac{r}{\mathcal{X_H}} \right)^{\nf{-\theta}{2}}r^{-k}\omega(r).
\end{align*}
Indeed, now $\Tilde{p}_k\in\mathcal{\overline{A}}_k$, because it holds that $\pi_k p_{k,r}\in\mathcal{\overline{A}}_k$. This reasoning alongside (\ref{thm21pfstep3omegaests}) and (\ref{thm21pfstep3midpt}) allows us to deduce that
\begin{equation*}
\begin{aligned}
    &r^{-k}\|u-\Tilde{p}_k \|_{\underline{L}^2(B_{r})}\leq C\int_r^{2r} t^{-k}\|u-\Tilde{p}_k \|_{\underline{L}^2(B_{t})} \frac{dt}{t} \\
    &\leq C\int_r^{2r} E_k(t) \frac{dt}{t} + C\int_r^{2r} t^{-k}\|\pi_kp_{k,t}-\pi_k\Tilde{p}_k\|_{\underline{
    L}^2(B_t)} \frac{dt}{t} + C\sum_{j=0}^{k-1}\int_r^{2r} t^{-k}\|\pi_j p_{k,t} \|_{\underline{L}^2(B_{t})} \frac{dt}{t} \\
    &\leq C\left(\frac{r}{\mathcal{X_H}} \right)^{\nf{-\theta}{2}}r^{-k}\omega(r).
\end{aligned}
\end{equation*}

Lastly, by utilizing the assumption (ii)$_k$, there exists now some corrector $\Tilde{\phi}_k\in\mathcal{A}_k$ corresponding to $\Tilde{p}_k$ that satisfies the respective estimate of (\ref{thmhighregclaim1}) such that $u-\Tilde{\phi}_k\in\mathcal{A}_{k-1}$. Furthermore, with the induction assumption (i)$_{k-1}$, there exists a harmonic polynomial $\Tilde{p}_{k-1}\in\mathcal{\overline{A}}_{k-1}$ for every $3^n\geq \mathcal{X_H}$ so that
\begin{equation*}
\begin{aligned}
    3^{-n}\overline{\lambda}^{1/2}\|u-\Tilde{\phi}_k-\Tilde{p}_{k-1}\|_{\underline{L}^2(\diamondsuit_{n})}&+3^{-ns}\left[\overline{\mathbf{A}}^{1/2}\begin{bmatrix}
\nabla u-\nabla \Tilde{\phi}_k-\nabla \Tilde{p}_{k-1} \\ \mathbf{a}\nabla (u-\Tilde{\phi}_k)-\overline{\mathbf{a}}\nabla\Tilde{p}_{k-1}
    \end{bmatrix} \right]_{\Hminusul(\diamondsuit_n)} \\
    &\leq C\left(\frac{\mathcal{X_H}}{3^n} \right)^{\nf{\theta}{2}} \|\mathbf{\overline{s}}^{1/2}\nabla \overline{u}\|_{\underline{L}^2(\diamondsuit_n)}.
\end{aligned}
\end{equation*}
This then implies the desired assertion of (\ref{thm21pfstep3goal}) by writing there that $\overline{u}:=\Tilde{p}_k+\Tilde{p}_{k-1}$ and utilizing the triangle inequality with our earlier reasoning.

\emph{Step 4. Proof of (iii)$_k^{'}$.} Let us now assume that claims (i)$_k$ and (ii)$_k$ are valid.
We fix $R\geq\mathcal{X_H}$ with $u\in\mathcal{A}(B_R)$ and $r_j:=\mu^{-j}\mathcal{X_H}$. Then, utilizing the same notation as in Step 3.1, the previous step guarantees, for every $u_{j+1}\in\mathcal{A}(B_{r_{j+1}})$, the existence of some $p_j \in\overline{\mathcal{A}}_k$ and $C(d,k,s,\gamma,\mu)<\infty$ such that
\begin{align}\label{thm21pfstep4harmappr}
    \inf_{p\in\overline{\mathcal{A}}_k}\|u_{j+1}-p\|_{\underline{L}^2(B_{r_j})}=\|u_{j+1}-p_j\|_{\underline{L}^2(B_{r_j})}\leq C
    \left(\left(\frac{r_j}{r_{j+1}} \right)^{k+1}+\left(\frac{r_j}{\mathcal{X_H}} \right)^{\nf{-\theta}{2}} \right)
    \|u_{j+1}\|_{\underline{L}^2(B_{r_{j+1}})}.
\end{align}
The inequality above follows from \eqref{thm21pfexcessestimate} as well as the simple observations of
\begin{align*}
    \inf_{p\in\overline{\mathcal{A}}_k}\|u_{j+1}-p\|_{\underline{L}^2(B_{r_j})}\leq \|u_{j+1}\|_{\underline{L}^2(B_{r_j})} \qand \|p_j\|_{\underline{L}^2(B_{r_j})}\leq 2\|u_{j+1}\|_{\underline{L}^2(B_{r_j})}.
\end{align*}
We may now refer to our assumption (ii)$_k$ to find $\phi_j\in\mathcal{A}_k$, which satisfies the bound that
\begin{align*}
    r_j^{-1}\overline{\lambda}^{1/2}\|\phi_j-p_j \|_{\underline{L}^2(B_{r_j})}
    \leq C\left(\frac{r_j}{\mathcal{X_H}} \right)^{\nf{-\theta}{2}}\| \overline{\mathbf{s}}^{1/2}\nabla p_j\|_{\underline{L}^2(B_{r_{j}})}.
\end{align*}
Now, by applying the Caccioppoli inequality and the triangle inequality, we are able to reason that
\begin{align*}
    \|\phi_j-p_j \|_{\underline{L}^2(B_{r_j})}
    \leq C\left(\frac{r_j}{\mathcal{X_H}} \right)^{\nf{-\theta}{2}}\| u_{j+1}\|_{\underline{L}^2(B_{r_{j+1}})}.
\end{align*}
Consequently, if we define that $u_j:=u_{j+1}-\phi_j$, then we may apply the triangle inequality after adding $p_j$ and $-p_j$ as $p_j\in\overline{\mathcal{A}}_k$ to deduce after reabsorbing that
\begin{align*}
    \|u_j\|_{\underline{L}^2(B_{r_j})}\leq C(\mu^{k+1}+\mu^{\nf{j\theta}{2}})\|u_{j+1}\|_{\underline{L}^2(B_{r_{j+1}})}.
\end{align*}
Lastly, a simple iteration scheme allows us to conclude for $\rho<\nf{\theta}{2}$ and sufficiently small $\mu$ that
\begin{align*}
    \|u_j\|_{\underline{L}^2(B_{r_j})}\leq C\left( \frac{r_j}{r_{J}}\right)^{k+1-\rho}\|u_{J}\|_{\underline{L}^2(B_{r_{J}})}
\end{align*}
after choosing such $J\in\N$ that $R\in [r_J, r_{J+1})$ as we fix $u_J=u$ to satisfy
\begin{align*}
    u_j=u-\sum_{i=j}^{J-1} \phi_i.
\end{align*}

We also note that if there is no $J$ with the aforementioned properties, then the claim holds by default since there would not be any available scales above $r$ in that case. This allows us to consider all of the scales $r\in[\mathcal{X_H},R]$ for which, by the arguments presented above, there exists $\phi_r\in\mathcal{A}_k$ so that
\begin{align*}
    \|\mathbf{s}^{1/2}\nabla(u-\phi_r)\|_{\underline{L}^2(B_r)}\leq C\left(\frac{r}{R} \right)^{k-\rho}\|\mathbf{s}^{1/2}\nabla u\|_{\underline{L}^2(B_R)}.
\end{align*}
Lastly, the triangle inequality alongside an iteration argument with respect to $t$ yields for every $t\in[\mathcal{X_H},\nf{r}{2}]$ that
\begin{equation*}
\begin{aligned}
    \|\mathbf{s}^{1/2}\nabla(\phi_{2t}-\phi_t)\|_{\underline{L}^2(B_r)}&\leq C\left(\frac{r}{t} \right)^{k-1}\left(\frac{t}{R} \right)^{k-\rho}\|\mathbf{s}^{1/2}\nabla u\|_{\underline{L}^2(B_R)} \\
    &= C\left(\frac{t}{r} \right)^{1-\rho}\left(\frac{r}{R} \right)^{k-\rho}\|\mathbf{s}^{1/2}\nabla u\|_{\underline{L}^2(B_R)}.
\end{aligned}
\end{equation*}
This provides the desired claim, since we can utilize the Haar measure and the triangle inequality for every $r\in[\mathcal{X_H},R]$ as $\phi:=\phi_\mathcal{X_H}$ to compute that
\begin{equation*}
    \begin{aligned}
\|\mathbf{s}^{1/2}\nabla(u-\phi)\|_{\underline{L}^2(B_r)}&\leq C\int_{\mathcal{X_H}}^{\nf{r}{2}} \|\mathbf{s}^{1/2}\nabla(\phi_{2t}-\phi_t)\|_{\underline{L}^2(B_r)} \frac{dt}{t} + \|\mathbf{s}^{1/2}\nabla(u-\phi_r)\|_{\underline{L}^2(B_r)} \\
&\leq C\left(\frac{r}{R} \right)^{k-\rho}\|\mathbf{s}^{1/2}\nabla u\|_{\underline{L}^2(B_R)},
    \end{aligned}
\end{equation*}
because we have the following evident upper bound for the integral over $t$ that
\begin{align*}
    \int_{\mathcal{X_H}}^{\nf{r}{2}}\left(\frac{t}{r} \right)^{1-\rho}\frac{dt}{t}\leq C<\infty.
\end{align*}

\emph{Step 5. Proof of (iii)$_k$.} Finally, we suppose that each of the assertions (i)$_{k+1}$, (ii)$_{k+1}$, and~(iii)$_{k+1}^{'}$ is valid. First, note that by (i)$_{k+1}$ and (ii)$_{k+1}$, we are able to identify the quotient spaces $\mathcal{A}_{k+1}/\mathcal{A}_{k}$ and $\overline{\mathcal{A}}_{k+1}/\overline{\mathcal{A}}_{k}$. This means that we can decompose the corrector $\phi\in\mathcal{A}_{k+1}$ into its lower-order parts. Thus, there exist such $\phi\in\mathcal{A}_k$ and $\widetilde{\phi}\in\mathcal{A}_{k+1}$ that each scale $r\in[\mathcal{X_H},R]$ satisfies the estimates that
\begin{align*}
    \|\mathbf{s}^{1/2}\nabla(u-\phi-\widetilde{\phi}) \|_{\underline{L}^2(B_r)}\leq C\left(\frac{r}{R} \right)^{k+1-\nf{\theta}{2}} \|\mathbf{s}^{1/2}\nabla u \|_{\underline{L}^2(B_R)}.
\end{align*}
as well as
\begin{align*}
    \|\mathbf{s}^{1/2}\nabla\widetilde{\phi} \|_{\underline{L}^2(B_r)}\leq C\left(\frac{r}{R} \right)^{k} \|\mathbf{s}^{1/2}\nabla u \|_{\underline{L}^2(B_R)}.
\end{align*}
Consequently, the desired claim now follows from a simple application of the triangle inequality.

\emph{Step 6. Proof of (\ref{thmhighregdims}).} The justification for the last assertion in (\ref{thmhighregdims}) remains exactly unchanged from the argument presented in \cite{akm}, but for the sake of completeness, we will briefly review it here as well. It is a well-known fact from the classical theory of harmonic functions (see, for example, \cite[Corollary 2.1.4]{armitage}) that the dimension $\dim(\overline{\mathcal{A}}_k)$ is given explicitly for every $k\in\N$ by
\begin{align*}
    \dim(\overline{\mathcal{A}}_k)=\binom{d+k-1}{k}+\binom{d+k-2}{k-1},
\end{align*}
where we interpret for $k=0$ that $\binom{d-2}{-1}=0$. Thus, it remains to argue inductively for $k$ that $\dim(\mathcal{A}_k)=\dim(\overline{\mathcal{A}}_k)$. The initial step for $k=0$ is already clear from the arguments above, since $\mathcal{A}_0=\overline{\mathcal{A}}_0$ is the set of constant functions. Next, as in the previous step, we note that results~(i)$_k$ and~(ii)$_k$ provide us with a canonical isomorphism between the quotient spaces $\mathcal{A}_k/\mathcal{A}_{k-1}$ and~$\overline{\mathcal{A}}_k/\overline{\mathcal{A}}_{k-1}$. This implies that their dimensions coincide, which, in turn, justifies the claim of~(\ref{thmhighregdims}) by the induction loop. This concludes the proof.
    \end{proof}
\end{theorem}

Quite naturally, similarly to the setting of uniformly elliptic equations (cf. \cite[Theorem 6.13]{ak}), there surely exists a local formulation for Theorem \ref{thmhighreg} above. However, proving this fact rigorously would mean that we have to repeat the same argument as above, but now in the local setting with some finite stopping scale $t\geq r\geq\mathcal{X_H}$. Furthermore, we would also require other results from \cite{ak}, but now in the high-contrast context. This local formulation would introduce for open and bounded sets $U\subset\R^d$ the local solution spaces $\mathcal{A}_k(U)$ that have $\ahom$-harmonic boundary values, or in other words,
\begin{align*}
    \mathcal{A}_k(U):=\{u\in\mathcal{A}(U)\; |\; u=p+H_{c}^1(U) \: \textnormal{for some} \: p\in\overline{\mathcal{A}}_k\},
\end{align*}
so that $\mathcal{A}_k$ is the limit of $\mathcal{A}_k(B_R)$ as $R\to\infty$.

Ideally, we would like to have this local version of the aforementioned result, which would provide a concrete tool for potential applications and calculations around this topic. In practice, this locality shows as the finiteness of each summation and integration. Unfortunately enough, these considerations will be left for future research projects.

\appendix

\section{Properties of harmonic polynomials}

In this appendix, we list some of the useful properties that the harmonic polynomials in the proof of Theorem \ref{thmhighreg} have. Especially, we are interested in the behavior of these functions within the adapted geometry in balls $B_r$ or cubes $\diamondsuit_n$ so that we can refer to the results and properties below in our earlier proofs. We will also consider the changes of variables needed to move between the adapted and Euclidean geometry. Our primary reference in this appendix is the book \cite{axler} by Axler, Bourdon, and Ramey.

Let us begin by recalling the definition of a \emph{harmonic polynomial}. A harmonic (real-valued) function $u\in C^2(U)$ in an open non-empty set $U\subset\R^d$ is a solution of the \emph{Laplace equation}
\begin{align}\label{appendixlaplace}
    \Delta u:=\sum_{j=1}^d u_{x_j x_j}=0.
\end{align}
We understand \emph{polynomials} as linear combinations of monomials, namely, a homogeneous polynomial $p$ of order $k\in\N$ has the form
\begin{align*}
    p(x)=\sum_{|q|=k}c_qx^q,
\end{align*}
where $c_q\in\R$ and $x\in\R^d$. The reader should familiarize themselves with the multi-index notation in which we write that $x^q:=x_1^q\ldots x_d^q$, $q!=q_1!\ldots q_d!$, and $|q|=q_1+\ldots+q_d$. Now, if $u$ is harmonic around the origin, we see by defining
\begin{align*}
    p_k(x):=\sum_{|q|=k} \frac{\partial_q u(0)}{q!}x^q
\end{align*}
that for points $x$ sufficiently close to the origin, it holds that
\begin{align}\label{appendixsphericalharmonics}    
    u(x)=\sum_{k=0}^\infty p_k(x).
\end{align}
Note that each $p_k$ here is a homogeneous polynomial of order $k$ and that the harmonicity of $u$ in $U$ implies (trivially) that $u$ is also differentiable in $U$. Since the Laplace operator $\Delta$ is linear, we see that each $p_k$ is also harmonic. Thus, we will call these functions \emph{harmonic polynomials}. Another important remark is that since (\ref{appendixlaplace}) holds, the degree of the harmonic polynomial determines its scaling properties due to the homogeneity property stating that $u(rx)=r^ku(x)$ as $r>0$.

There are many useful properties that harmonic polynomials have, so let us present a few of these results next. First, we should note that we can write every polynomial in $\R^d$ as a properly normalized sum of low-order harmonic polynomials. Furthermore, for a polynomial~$p$ of the $k$th order defined in $U\subset \R^d$, we can uniquely write that
\begin{align}
    p=\sum_{j=0}^k p_j,
\end{align}
where each $p_j$ is a homogeneous polynomial on $U\subset \R^d$ with a degree of $j=0,\ldots ,k$. We may identify these polynomials $p_j$ as the \emph{homogeneous parts} of $p$ for each degree. As hinted earlier, we quickly note that $p$ itself is harmonic if and only if all functions $p_j$ are harmonic.

Perhaps the most important observation (at least for our purposes) is that we can decompose the space $L^2(U)$ as a Hilbert space by \emph{spherical harmonics}. Namely, let us denote the set of all $k$-order harmonic polynomials on $U\subset\R^d$ by $\mathcal{H}_k(U)$. An important special case here is the space $\mathcal{H}_k(S)$ for the unit sphere $S$ of $\R^d$, and we call the restriction $p_{|S}$ of a harmonic polynomial $p\in\mathcal{H}_k(\R^d)$ the $k$th degree spherical harmonics of $p$. Consequently, \cite[Proposition 5.9]{axler} now tells us that $\mathcal{H}_k(S)$ and $\mathcal{H}_l(S)$ are always orthogonal to each other in $L^2(S)$ as $k\neq l$. Furthermore, if $p$ is any $k$-degree polynomial in $\R^d$, then the restriction map $p_{|S}$ can be expressed as a summation of at most $k$-degree spherical harmonics. Based on these observations, we can finally state the important result that
\begin{align*}
    L^2(S)=\bigoplus_{k=0}^\infty \mathcal{H}_k(S),
\end{align*}
or, in other words, the Hilbert space $L^2(S)$ is the direct sum space of the different order harmonic polynomial spaces $\mathcal{H}_k(S)$.

Let us next briefly recall the definition and properties of these infinite direct sums of Hilbert spaces. For a Hilbert space $H$, we may write that
\begin{align*}
    H=\bigoplus_{k=0}^\infty H_k
\end{align*}
when the following requirements are met.
\begin{enumerate}
    \item Each $H_k$ is a closed subspace of $H$.
    \item For all $k\neq l$, the spaces $H_k$ and $H_l$ are orthogonal to each other.
    \item For every element $x\in H$, there exists (unique) $x_k\in H_k$ so that
    \begin{align*}
x=\sum_{k=0}^\infty x_k,
    \end{align*}
    where the sum on the right-hand side always converges within the norm of $H$.
\end{enumerate}
In the case where the aforementioned axioms are satisfied, we say that $H$ is a direct sum of the spaces~$H_k$. Note also that the linear span of the union $\bigcup_k H_k$ is dense in $H$ due to the third condition.

Our next goal here is to study the linear transformations from the adapted geometry introduced in Assumption \ref{d.renormellipt} or in (\ref{defadaptedballs}) to the standard Euclidean geometry and vice versa. In other words, we wish to understand the change of variables needed to change the current geometry to the other. Let us focus on the case where we perform the change of variables from the adapted geometry $\diamondsuit_n$ to the Euclidean geometry $\square_n$ here.

Suppose that $u\in H^1(\diamondsuit_n)$ satisfies the homogenized elliptic equation
\begin{align}\label{appendixellipeq}
    -\nabla\cdot\ahom \nabla u = 0
\end{align}
in $\diamondsuit_n$ for some fixed $n\in\N$. We shall make the following change of variables so that for each $x\in\diamondsuit_n$, we have that $u(x)=v(\mathbf{q}_0^{-1}x)$ for some function $v\in H^1(\square_n)$ and $\mathbf{q}_0$ as in Assumption~\ref{d.renormellipt}. Note that the approximation $\mathbf{q}_0\approx \overline{\lambda}^{-1/2}\overline{\mathbf{s}}^{1/2}$ is most often accurate enough for us.
Taking the gradient of both sides above provides that $\nabla u(x)=\mathbf{q}_0^{-1}\nabla v(\mathbf{q}_0^{-1}x)$, which reduces (by taking the divergence) the equation of (\ref{appendixellipeq}) to a simple Laplace equation $\Delta v=0$ within the standard Euclidean geometry of~$\square_n$. Of course, similar reasoning in the opposite direction creates a pathway for us to transfer objects from the Euclidean geometry to $\diamondsuit_n$. However, we have been utilizing this direction throughout the paper, so we will not present the details here. Indeed, it is a routine exercise to show that these results presented above still hold after this change of variables.

Let us still point out a couple of useful properties and characteristics that harmonic polynomials have. First, we have the useful identity for all harmonic polynomials $u$ in the Euclidean balls $B_s^e$ and~$B_r^e$ that
\begin{align*}
    \|u\|_{\underline{L}^2(B_s^e)}^2=\sum_{j=0}^k \|p_j\|_{\underline{L}^2(B_s^e)}^2=\sum_{j=0}^k \left(\frac{s}{r}\right)^{2j} \|p_j\|_{\underline{L}^2(B_r^e)}^2
\end{align*}
as $0<s\leq r$. Secondly, if it holds that $\square_n\subset B_r^e$ and $B_s^e\subset \square_{n-1}$, then we may reason that
\begin{align*}
    \|u\|_{\underline{L}^2(\square_n)}^2\leq C\|u\|_{\underline{L}^2(B_r^e)}^2\leq C\left(\frac{r}{s} \right)^{2\textnormal{deg}\{u\}}\|u\|_{\underline{L}^2(B_s^e)}^2\leq C\left(\frac{r}{s} \right)^{2\textnormal{deg}\{u\}}\|u\|_{\underline{L}^2(\square_{n-1})}^2.
\end{align*}
Naturally, the same results are valid for adapted balls $B_s$ and $B_r$ as well by applying the aforementioned change of variables by $\mathbf{q}_0$.

\bigskip

\noindent \textbf{Acknowledgments:} The author was supported by the Research Council of Finland and the European Research Council (ERC) under the Horizon 2020 research and innovation program of the European Union (grant agreement No. 818437). The author is extremely grateful to his Ph.D. advisor, Prof. Tuomo Kuusi, for multiple helpful discussions and improvements to the manuscript.

{\footnotesize
\bibliographystyle{alpha}
\bibliography{highcontrast}
}

\end{document}